\documentclass[10pt]{amsart}

\usepackage{amsthm,amsmath,amssymb,epsfig,graphicx,latexsym,amsthm,mathdots}

\usepackage{graphicx,amsfonts,subfigure}

\usepackage{hyperref}
\numberwithin{equation}{section}

\usepackage[usenames, dvipsnames]{color}

\theoremstyle{plain}
\newtheorem{theorem}{Theorem}[section]
\newtheorem{lemma}[theorem]{Lemma}
\newtheorem{proposition}[theorem]{Proposition}
\newtheorem{corollary}[theorem]{Corollary}

\theoremstyle{remark}
{
    \newtheorem{definition}[theorem]{Definition}
    
    \newtheorem{remark}[theorem]{Remark}

}

\def\dd{{\rm d}}
\newcommand{\ddt}{\frac{\rm d}{{\rm d} t} }

\def\weight(#1,#2){c_{#1,#2}}

\def\ra{{\rm ra}}

\if {

}{{\caption}}
} \fi

\def\uh{\hat{u}}

\def\xh{\hat{x}}

\def\hb{\bar{h}}

\def\vb{\bar{v}}

\def\xb{\bar{x}}
\def\yb{\bar{y}}
\def\zb{\bar{z}} 

\def\C{\mathcal{C}}

\def\LL{\mathcal{L}}
\def\M{\mathcal{M}}

\def\P{\mathcal{P}}

\def\Z{\mathcal{Z}}

\def\eps{\varepsilon}

\def\Om{{\Omega}}

\def\1B{{\bf  1}}

\def\dist{\mathop{\rm dist}}

\def\intt{\mathop{\rm int}}

\def\supp{\mathop{\rm supp}}

\def\Ker{\mathop{\rm Ker}}

\def\Max{\mathop{\rm Max}}

\def\half{\mbox{$\frac{1}{2}$}}

\def\1B{{\bf  1}}

\newcommand{\RR}{\mathbb{R}}

\def\cN{\mathbb{N}}

\def\cR{\mathbb{R}}

\newcommand\be{\begin{equation}}
\newcommand\ee{\end{equation}}
\newcommand\ba{\begin{array}}
\newcommand\ea{\end{array}}
\newcommand{\bea}{\begin{eqnarray}}
\newcommand{\eea}{\end{eqnarray}}
\newcommand{\bean}{\begin{eqnarray*}}
\newcommand{\eean}{\end{eqnarray*}}

\def\rar{\rightarrow}

\def\ds{\displaystyle}
\def\disp{\displaystyle}

\def\la{\langle}
\def\ra{\rangle}

\newcommand{\umin}{ u_{\rm min}}
\newcommand{\umax}{ u_{\rm max}}
\def\benl{\begin{equation*}}
\def\eenl{\end{equation*}}

\newcommand{\dtt}{ \mathrm{d}t}
\newcommand{\xib}{ \bar\xi}

\DeclareMathAlphabet{\mathpzc}{OT1}{pzc}{m}{it}

\begin{document}

\title[Control-affine problems with state constraints]{Second order analysis of control-affine problems with scalar state constraint}
\author[M.S. Aronna and F. Bonnans and B.S. Goh]
{M. Soledad Aronna\address{EMAp/FGV, 22250-900 Rio de Janeiro, Brazil}\email{soledad.aronna@fgv.br}
and
J. Fr\'{e}d\'{e}ric Bonnans\address{INRIA-Saclay and Centre de 
Math\'ematiques Appliqu\'ees, Ecole Polytechnique, 91128 Palaiseau, France}
\email{Frederic.Bonnans@inria.fr}  
and 
Bean San Goh\address{Curtin Sarawak Research Institute, 98009 Miri, Sarawak, Malaysia}\email {bsgoh@curtin.edu.my}
}

\thanks{This article will appear in {\em Mathematical Programming - Series A}}

\maketitle

\begin{abstract}
In this article we establish new second order necessary and sufficient
optimality conditions for a class of control-affine problems with a
scalar control and a scalar state constraint. These optimality
conditions extend to the constrained state framework the Goh
transform, which is the classical tool for obtaining an extension of
the Legendre condition. 
\if{
We propose a shooting algorithm to solve numerically this class of
problems and we provide a sufficient condition for its local
convergence. We provide examples to illustrate the theory.
} \fi
\end{abstract}

\section{Introduction}


Control-affine problems have been intensively studied since the 1960s and there is a wide literature on this subject. In what respect to second order conditions, the main feature of these systems that are affine in the control variable is that the second derivative of the pre-Hamiltonian function with respect to the control vanishes, and hence the classical {\em Legendre-Clebsch conditions} hold trivially and do not provide any useful information. Second and higher order necessary conditions for problems affine in the control, without control nor state constraints were first established in \cite{Kel64,Goh66,GabKir72,JacSpe71,AgSachkov}. \if{ AgrGam76,Goh66aKelKopMoy67,} }\fi  The case with control constraints and purely bang-bang solutions was investigated by \cite{MilOsm98,AgrSteZez02,PogSpa11,MauOsm03b,Fel04,MR3012263} 
\if{Osm04,Fel03, ,Fel05 MauOsm03,}\fi
among many others, while the class of bang-singular solutions was analyzed in e.g. \cite{PogSte11,ABDL11,FraTon13}.

This article is devoted to the study of Mayer-type optimal control problems governed by the dynamics
\benl
\dot x_t = f_0(x_t) + u_t f_1(x_t),\quad \text{for a.a. } t\in [0,T],
\eenl
subject to endpoint constraints
$$
\Phi(x_0,x_T) \in K_\Phi,
$$
control constraints
$$
u_{\rm min} \leq u_t \leq u_{\rm max},
$$
and a scalar state constraint of the form
$$
g(x_t) \leq 0.
$$
For this class of problems,  we show necessary optimality conditions
involving the regularity of the control and the state constraint
multiplier at the junction points. Some of these necessary conditions
which hold at the junction points were proved in
\cite{Maurer77}. Moreover, we provide second order necessary and
sufficient optimality conditions in integral form obtained through the
{\em Goh transformation} \cite{Goh66}. 
\if{
We also propose a shooting-like numerical scheme and we show a
sufficient condition for its local quadratic convergence, that is also
a second order sufficient condition for optimality (in some sense to
be specified later on). Finally, we solve numerically an example of
practical interest.
} \fi

This investigation is strongly motivated by applications since it allows to deal with both control and state constraints, which appear naturally in realistic models. Many practical examples can be found in the existing literature, a non exhaustive list includes the prey-predator model \cite{GLV74}, the Goddard problem in presence of a dynamic pressure limit \cite{SeyCli93,GraPet08}, an optimal production and maintenance system studied in \cite{MauKimVos05}, and a recent optimization problem of running strategies \cite{AftBon14}. 
We refer also to \cite{dePinho2005}, \cite{Bonnard2003},  \cite{Schaettler2006} and references therein.

About second order analysis in the state constrained case,
we quote the early work by Russak
\cite{MR0397507,MR0390876},
and more recently,
 Malanow\-ski and Maurer \cite{MalMau96}, Bonnans and Hermant \cite{BonHer09} provided second order necessary and sufficient optimality conditions and related the sufficient conditions with the convergence of the shooting algorithm in the case where the strengthened Legendre-Clebsch condition holds.
Second order necessary conditions for the general nonlinear case with phase constraints were also proved in \cite{MR780283,MR2376656}.
For control-affine problems with bounded scalar control variable, a scalar state constraint (as is the case in this article) and solutions (possibly) containing singular, bang-bang and constrained arcs, Maurer \cite{Maurer77} proved necessary conditions (similar to those developed by McDanell and Powers \cite{MacPow71}) that hold at the junction points of optimal solutions.
In Maurer et al. \cite{MauKimVos05} they extended to the
state-constrained framework, a second order sufficient test for
optimality given in Agrachev et al. \cite{AgrSteZez02} and
Maurer-Osmolovskii \cite{MauOsm05} for optimal bang-bang solutions.
Also in \cite{Schaettler2006}, the author provided a synthesis-like method to prove optimality for a class of control-affine problems with scalar control, vector state constraint and bang-constrained solutions. Some remarks about how the method in \cite{Schaettler2006} could be extended to general bang-singular-constrained solutions are given in \cite{Schaettler2005}, but no proof of the validity of this extension is provided.

There is a literature dealing with the important case when the
standard statement of Pontryagin's principle is degenerate in the
sense that the costate is equal to zero, so that the principle is trivially verified and no information of the optimal solution is provided. In some cases, a nontrivial version of 
Pontryagin's principle can be obtained, see in particular 
Arutyunov \cite{MR1845332},
Rampazzo and Vinter \cite{MR1814264}. 
Also second order conditions for the general nonlinear case were
stated in 
strongly non-degenerate form, see e.g. \cite{MR2376656,MR1845332}.

\if{
As it is commonly known nowadays, the application of the necessary conditions provided by Pontryagin's Maximum Principle leads to an associated two-point boundary-value problem (TPBVP) for the optimal trajectory and its corresponding multiplier. A natural way for solving numerically TPBVPs is the application of {\em shooting-like algorithms.} This type of algorithms has been used extensively to solve optimal control problems (see e.g. \cite{Bul71,Pes94} and references therein) and, in particular, it has been applied to control-affine problems both with and without state constraints. Maurer \cite{Mau76} proposed a shooting scheme for solving a problem  with bang-singular solutions, which was generalized quite recently by Aronna, Bonnans and Martinon in \cite{AroBonMar11}, where they provided a sufficient condition for its local convergence.
Both these articles \cite{Mau76} and \cite{AroBonMar11} analyze the case when no state constraints are present.
 Practical control-affine problems with state constraints were solved
 numerically in several articles, a non extensive list includes Maurer
 and Gillessen \cite{MauGil75}, Oberle \cite{Obe79} and Fraser-Andrews
 \cite{Fra89}.
} \fi

Up to our knowledge, there is no result in the existing literature
about second order necessary and sufficient conditions in integral
(quadratic) form for control-affine problems with state constraints
and solutions containing singular arcs.
\if{
(ii) sufficient conditions for the convergence of shooting algorithms
in this context.
} \fi

The paper is organized as follows. In Section \ref{SecFramework} we give the basic definitions and show necessary optimality conditions concerning the regularity of the optimal control and associated multipliers. Section \ref{SecSOC} is devoted to second order necessary optimality conditions in integral form and to the Goh transformation, while second order sufficient conditions are provided in Section \ref{SecSC}.
 \if{
A shooting-like method and a sufficient condition for its local
quadratic convergence are given in Section \ref{SecShooting}. This
algorithm is implemented in Section \ref{SecExamples} to solve
numerically a variation of the regulator problem. 
} \fi
The Appendix is consecrated to the presentation of abstract results on second order necessary conditions.
 
 \vspace{4pt}

\noindent {\bf Notations.} 
Let $\cR^k$ denote the $k-$dimensional real space, i.e. the space of column real vectors of dimension $k,$ and by  $\cR^{k*}$ its corresponding dual space, which consists of $k-$dimensional row real vectors. With $\cR^k_+$ and $\cR^k_-$ we refer to the subsets of $\cR^k$ consisting of vectors with nonnegative, respectively nonpositive, components.
We write $h_t$ for the value of function $h$ at time $t$ if $h$ is a function that depends only on $t,$ and by $h_{i,t}$ the $i$th component of $h$ evaluated at $t.$
Let $h(t+)$ and $h(t-)$ be, respectively, the right and left limits of $h$ at $t,$ if they exist.
Partial derivatives of a function $h$ of $(t,x)$ are referred as $D_th$ or $\dot{h}$ for the derivative in time, and $D_xh,$ $h_x$ or $h'$ for the differentiations with respect to space variables. The same convention  is extended to higher order derivatives.
By $L^p(0,T)^k$ we mean the Lebesgue space with domain equal to the interval $[0,T]\subset \cR$ and with values in $\cR^k.$ The notations $W^{q,s}(0,T)^k$ and $H^1(0,T)^k$ refer to the Sobolev spaces (see Adams \cite{Ada75} for further details on Sobolev spaces).
We let $BV(0,T)$ be the set of functions with bounded total variation. 
In general, when there is no place for confusion, we omit the argument $(0,T)$ when referring to a space of functions. For instance, we write $L^\infty$ for $L^\infty(0,T),$ or $(W^{1,\infty})^{k*}$ for the space of $W^{1,\infty}-$functions from $[0,T]$ to $\cR^{k*}.$
We say that a function $h: \cR^k \to \cR^d$ is of class $C^\ell$ if it is $\ell-$times continuously differentiable in its domain.

\if{

\bigskip


{\bf References to be included}. 

\begin{itemize}
\item Second order necessary conditions: abstract results \cite{PAOP}
(Kawasaki, Cominetti).
\item {\em Geometric approach:} Bonnard and Kupka \cite{MR1441391}:
optimality for the minimal time problem in the absence of 
conjugate point.
\item {\em Others to look at:} Hoehener \cite{Hoe12},  Fraser-Andrews \cite{Fra96}, ...
\end{itemize}

\bigskip

}\fi

\section{Framework}\label{SecFramework}

\subsection{The problem}
Consider the control and state spaces $L^\infty$ and $(W^{1,\infty})^n,$ respectively. We say that a control-state pair $(u,x)\in L^\infty\times(W^{1,\infty})^n$
is a  {\em trajectory} if it satisfies both the {\em state equation}
\be
\label{bsbstateeq}
\dot x_t = f_0(x_t) +u_{t} f_1(x_t), \quad \text{for a.a. } t\in [0,T],
\ee
and the finitely many  {\em endpoint constraints}
of equality and inequality type
\be
\Phi(x_0,x_T) \in K_\Phi.
\ee
Here $f_0$ and $f_1$ are twice continuously differentiable and Lipschitz continuous vector fields over $\RR^n$,
$\Phi$ is of class $C^2$ from $\RR^n\times\RR^n$ to $\RR^{n_1+n_2},$ 
and
\be
K_\Phi := \{ 0 \}_{\RR^{n_1}} \times \RR^{n_2}_-,
\ee
where $\{ 0 \}_{\RR^{n_1}}$ is the singleton consisting of the zero vector of $\RR^{n_1}$ and $\RR^{n_2}_- := \{y\in \RR^{n_2}: y_i\leq 0,\, \text{for all } i=1,\dots,n_2\}.$
Given $(u,x_0)\in L^\infty\times\RR^n$, 
\eqref{bsbstateeq} has a unique solution. In addition, we consider
the {\em cost functional} $\phi:\cR^n \times \cR^n \to \cR,$ then
 {\em bound control constraints} 
\be
\label{bsbcc} 
\umin\leq u_{t} \leq \umax, \quad \text{for a.a. } t\in [0,T],
\ee
and a {\em scalar state constraint} 
\be
\label{stateconstraint1}
g(x_t) \leq 0,\quad \text{for all } t\in [0,T],
\ee
with $\phi$ and $g:\RR^n\rar \RR$ of class $C^2.$ 
Here we allow $u_{\min}$ and $u_{\max}$ to be either finite real numbers, or to take the values  $-\infty$ or $+\infty,$ respectively,  in the sense that problems with control constraints of the form $u_t \leq u_{\max}$ or $u_{\min} \leq u_t$ are also considered in our investigation, as well as problems in the absence of control constraints.
We say that the trajectory $(u,x)$
is {\em feasible} if it satisfies 
\eqref{bsbcc}-\eqref{stateconstraint1}.
Let us then consider the optimal control problem in the Mayer form
\be
\label{P}\tag{P}
\min \phi(x_0,x_T);
\qquad \text{subject to \eqref{bsbstateeq}-\eqref{stateconstraint1}.}
\ee

\subsection{Regular extremals}

We set $f(u,x) := f_0(x)+u f_1(x)$,
and define the {\em pre-Hamiltonian function}
and the {\em endpoint Lagrangian,} respectively, by 
\be
\label{defpreham}
\left\{ 
\begin{split}
H(u,x,p) &:= p f(u,x) = p \big(f_0(x)+u f_1(x) \big),
\\
\ell^{\beta,\Psi}(x_0,x_T) &:= \beta \phi(x_0,x_T) + \Psi \Phi(x_0,x_T),
\end{split}
 \right.
\ee
where $p\in \RR^{n*}$, $\beta\in\RR$ and $\Psi \in \RR^{(n_1+n_2)*}.$

Any function $\mu \in BV(0,T)$ (shortly, $BV$) has left limit on $(0,T]$ and right limits on
$[0,T)$ and, therefore, the values $\mu_{0+}$ and $\mu_T$ are well-defined. Moreover, $\mu$ has a distributional derivative that belongs to the space $\M(0,T)$ (shortly, $\M$) of finite Radon measures. Conversely, any measure $\dd\mu\in \M$ can be identified with the derivative of a function $\mu$ of bounded variation such that $\mu_T \in BV_0,$ i.e., $\mu$ belongs to the {\em space of bounded variation functions that vanish at time $T.$}  

Let $(u,x)$ be a feasible trajectory. 
We say that $\Psi \in \RR^{(n_1+n_2)*}$ is 
{\em complementary to the endpoint constraint} if
\be
\label{lmscac1}
\Psi \in N_\Phi \big( \Phi(x_0,x_T) \big),
\ee
where $N_\Phi \big( \Phi(x_0,x_T) \big)$ denotes the {\em normal cone to $K_\Phi$} at the point $\Phi(x_0,x_T),$
i.e.
\be
\label{NPhi}
N_\Phi \big( \Phi(x_0,x_T) \big)
:= \{ \Psi \in \RR^{(n_1+n_2)*} :  \Psi_i \geq 0, \; 
\Psi_i \Phi_i(x_0,x_T)=0, \, i=n_1+1,\ldots,n_1+n_2\}.
\ee
A bounded variation function $\mu$ is 
{\em complementary to the state constraint} if and only if
\be
\label{mucomsc}
\dd  \mu \geq 0, \quad \text{and} \quad 
\int_0^T g(x_t) \dd \mu_{t} =0.
\ee 
For $\beta\in \RR,$  $\Psi\in \RR^{(n_1+n_2)*}$ and 
$\mu\in BV_0,$
the {\em costate equation} associated with $(\beta,\Psi,\dd\mu)$ is given by
\be
\label{costateeq}
- \dd p_t =
p_t f_x(u_t,x_t) \dd t + g'(x_t) \, \dd \mu_{t},
\quad \text{for a.a. } t\in [0,T],
\ee
with endpoint conditions
\be
\label{lmscac7}
(- p_{0},p_{T}) = D\ell^{\beta,\Psi} (x_0,x_T).
\ee
Given $(\beta,\Psi, \dd\mu) \in\RR\times \RR^{(n_1+n_2)*} \times \M,$ the boundary value problem \eqref{costateeq}-\eqref{lmscac7} has at most one solution.
In addition, the {\em condition of minimization} of 
the pre-Hamiltonian $H$  implied by the Pontryagin's Maximum Principle
 can be expressed as follows,
for a.a. $t \in [0,T],$
\be
\label{lmscac4a}
\left\{ \ba{lll}
u_t = \umin, & \text{if $\umin > -\infty$ and $p_t f_1(x_t) >0$,}
\\
u_t = \umax, & \text{if $\umax < +\infty$ and $p_t f_1(x_t) <0,$}
\\
p_tf_1(x_t) = 0, &\text{if } \umin < u_t < \umax.
\ea\right.
\ee
Denote the quadruple of dual variables by 
$\lambda := (\beta,\Psi,p,\dd\mu)$,
element of the space
\be
E^\Lambda:=\RR_+\times\RR^{(n_1+n_2)*}\times BV^{n*} \times
\M.
\ee
The {\em Lagrangian} of the problem is
\be
\label{lmscac5}
\LL(u,x,\lambda):=\ell^{\beta,\Psi}(x_0,x_T) 
+ \int_0^T p_t \big(f(u_t,x_t) - \dot x_t \big) \dd t 
+ \int_0^T g(x_t) \dd \mu_{t}.
\ee
Note that the costate equation \eqref{costateeq} expresses the stationarity
of $\LL$ with respect to the state. 
For a feasible trajectory $(u,x) \in L^\infty \times (W^{1,\infty})^n,$ define the {\em set of Lagrange multipliers}  as
\be
\label{Lambda}
\Lambda(u,x):= \left\{ \ba{lll}
\lambda = (\beta,\Psi,p,\dd\mu) \in E^\Lambda: \;
(\beta,\Psi,\dd\mu)\neq 0; 
\\
\text{\eqref{bsbstateeq}-\eqref{stateconstraint1} and
  \eqref{lmscac1}-\eqref{lmscac4a} hold} 
\ea\right\}.
\ee
\begin{remark}
\label{remPL.r}
We could as well call $\Lambda(u,x)$ the set of Pontryagin multipliers
since they satisfy  condition of minimization of the Hamiltonian. 
But since the dynamics is an affine function of the control variable,
the condition of minimization of the Hamiltonian is equivalent to the 
stationarity condition 
\eqref{lmscac4a} which gives the classical definition of
Lagrange multipliers in this setting. 
\end{remark}
In addition we set, for $\beta\in \RR_+,$
\be
\label{LambdaBeta}
\Lambda_\beta(u,x) := \{ (\beta,\Psi,p,\dd\mu)\in \Lambda(u,x)
\}. 
\ee
Since $\Lambda(u,x)$ is a cone we have that 
\be
\Lambda(u,x) = 
\left\{ \ba{lll}
\RR_+ \Lambda_1(u,x), &
\text{ if $\Lambda_0(u,x)=\emptyset$,}
\\ 
\Lambda_0(u,x)  + \RR_+ \Lambda_1(u,x),
&\text{ otherwise.}
\ea\right.
\ee

Consider a nominal feasible trajectory $(\uh,\xh) \in L^\infty \times (W^{1,\infty})^n.$ Set
\be
\label{AB}
A_t := D_x f(\uh_t,\xh_t),\quad \text{for }t\in [0,T].
\ee
From now on, when the argument of a function is omitted, we mean that
it is evaluated 
at the nominal pair $(\uh,\xh).$ In particular, we write $\Lambda$ to refer to $\Lambda(\uh,\xh).$

For $(v,z^0) \in L^\infty\times \RR^n$, let $z[v,z^0] \in (W^{1,\infty})^n$ denote the solution 
of the {\em linearized state equation}
\be
\label{lineq}
\dot z_t = A_t z_t + v_t f_{1,t}, \quad \text{for a.a. }t\in [0,T],
\ee
with initial condition
\be
\label{lineq0}
z_0 = z^0.
\ee

Let $\bar\Phi_E$ denote the function from $L^\infty \times \cR^n$ to $\cR^{n_1}$ that,
to each 
$(u,x_0) \in L^\infty \times \cR^n,$  assigns the value $\big(\Phi_1(x_0,x_T),\dots,\Phi_{n_1}(x_0,x_T) \big),$ where $x$ is the solution of \eqref{bsbstateeq} associated with $(u,x_0).$
For some results obtained in this article, we shall consider the following {\em qualification condition,} which corresponds to {\em Robinson condition \cite{Rob76a}} (see also \eqref{qualif} and Remark \ref{tool_ocfo.r3} in the Appendix):
 \be
\label{tool_ocfo.qualifoc}
\left\{ \ba{ll}
{\rm (i)} & D \bar\Phi_E(\uh,\xh_0)\text{ is onto }  \text{from } L^\infty \times \cR^n \text{ to } \RR^{n_1},
\\
{\rm (ii)} & \text{there exists } (\vb,\zb^0)\in L^\infty \times \cR^n \text{ in the kernel of 
$D \bar\Phi_E(\uh,\xh_0)$,} 
\\ & \text{such that, for some } \eps>0, \text{ setting } \zb=z[\vb,\zb^0], \text{ one has:}
\\ &     \uh_t + \vb_t \in  [\umin+\eps,\umax-\eps], \, 
            \text{ for  a.a. $t \in [0,T],$}
\\ & 
           g(\xh_t) + g'(\xh_t)\zb_t < 0, \text{ for all } t\in [0,T], 
\\ & 
            \Phi'_i(\xh_0,\xh_T)(\zb^0,\zb_T) <0, \text{ if }
            \Phi_i(\xh_0,\xh_T) =0, \text{ for } n_1+1\leq i \leq n_1+n_2.
\ea\right.
\ee 
Here the notation $\uh_t+\vb_t \in  [\umin+\eps,\umax-\eps]$ means
that, a.e., 
$\uh_t+\vb_t \in  [\umin+\eps,+\infty)$ if $\umax=+\infty$ and $\umin$ is finite; $\uh_t+\vb_t \in  (-\infty,\umax-\eps)$ if $\umin=-\infty$ and $\umax$ is finite; and $\uh_t+\vb_t \in \cR$ if neither $\umin$ nor $\umax$ is finite.

\begin{definition}
A {\em weak minimum} for \eqref{P} is a feasible trajectory $(u,x)$ such that $\phi(x_0,x_T) \leq \phi(\tilde x_0,\tilde x_T)$ for  any feasible $( \tilde u,\tilde x)$ for which $\|(\tilde u,\tilde x)-(u,x)\|_{\infty}$ is small enough.
\end{definition}

\begin{theorem}
Assume that $(\uh,\xh) $ is a weak minimum for \eqref{P}.
Then 
{\rm (i)} 
the set $\Lambda$
of Lagrange multipliers is nonempty,
{\rm (ii)} 
if the qualification condition 
\eqref{tool_ocfo.qualifoc} holds, then 
$\Lambda_0$ is empty, and 
$\Lambda_1$ is bounded and weakly* compact.
\end{theorem}

\begin{proof}
If \eqref{tool_ocfo.qualifoc}(i) holds,
we deduce the result from proposition
\eqref{propderiv.p}.
Otherwise, 
there exists some nonzero $\Psi_E$ in the image of 
normal space to 
$D \bar\Phi_E(\uh,\xh_0)$. Setting $\Psi_i=0$
for all $n_1<i\leq n_1+n_2$, and $p=0$,
we obtain a (singular) Lagrange multiplier. 
This ends the proof.
\end{proof}

\subsection{Jump conditions}

Given a function of time $h:[0,T] \to \cR^d$ for $d\in \cN,$ 
we define its {\em jump at time $t\in [0,T]$} by 
\be
[h_t] := h(t+) - h(t-),
\ee
when the left and right limits,
$h(t-)$ and $h(t+),$ respectively, exist and are finite. Here we adopt  the convention 
$h(0-) :=h_0$  and $h(T+) :=h_T$. 
For any function of bounded variation the associated jump function is well-defined.
For a function defined almost everywhere with respect to the Lebesgue measure, we will accord that its jump at time $t$ is the jump at $t$ of a representative of this function for which the left and right limit exist, provided that such a representative exists.
By the costate equation \eqref{costateeq},
we have that, for any $(\beta,\Psi,p,\dd \mu)\in \Lambda,$ 
\be
\label{stateconstraint4}
\left\{ \ba{ll} 
{\rm (i) } & \disp
[ p_t ] = - [ \mu_t ] g'(\xh_t),
\vspace{1mm}
\\
{\rm (ii) } & \disp
 \big[ H_{u}(t) \big] = [ p_t]  f_1(\xh_t) = - [ \mu_t ] g'(\xh_t) f_1(\xh_t).
\ea\right.
\ee 
In addition, if 
$[\uh_t]$ makes sense, 
then the jump in the derivative of the state constraint exists and
satisfies
\be
\label{uhuk.l1_1}
 \left[ \ddt g(\xh_t) \right] =  [ \uh_t ] g'(\xh_t) f_1(\xh_t).
 \ee

\begin{lemma}
\label{uhuk.l3vh}
Let $(\beta,\Psi,p,\dd \mu)\in \Lambda.$ Then, if $t\in [0,T]$ is such that $[H_u(t)]=0$ and $g'(\xh_t) f_1(\xh_t) \neq 0,$ then $\mu$ is continuous at $t.$
\end{lemma}

\begin{proof}
It follows from item (ii) in equation \eqref{stateconstraint4}.
\end{proof}

\begin{lemma}
\label{uhuk.l1}
Let $t\in (0,T)$ be such that 
$[\uh_t]$ makes sense. Then the following conditions hold
\be
\label{uHu}
\left\{ 
\begin{split}
{\rm (i) } \quad &
[ \uh_t ]   [ H_u(t)] =0,
\\
{\rm (ii) } \quad &
[\mu_t] \left[ \ddt g(\xh_t) \right] =0.
\end{split}
\right. 
\ee
\end{lemma}

\begin{proof}
Note that 
\be
\label{mudotg}
[\mu_t] \left[ \ddt g(\xh_t) \right] = [\mu_t] g'(\xh_t) f_1(\xh_t) [\uh_t]
= -[H_u(t)] [\uh_t],
\ee
where the last equality follows from \eqref{stateconstraint4}. This
implies that (i) is equivalent to (ii),
and so we only need to prove (i), which holds trivially when $[\mu_t] = 0$.
Hence, let us assume that $[\mu_t] \neq 0$.
Then $g(\xh_t)=0$ in view of \eqref{mucomsc} and, necessarily, 
$t\mapsto g(\xh_t)$ attains its maximum at $t$, so that $\left[ \ddt g(\xh_t) \right] \leq 0.$ 
Since $\dd\mu \geq0,$ it follows that $[\mu_t] \geq 0$,
and therefore by \eqref{mudotg}, $[H_u(t)] [\uh_t] \geq 0.$ 
However, the converse inequality holds in view of \eqref{lmscac4a}.
The conclusion follows.
\end{proof}

We say that the {\em state constraint is of first order} if
\be
\label{1order}
g'(\xh_t)f_1(\xh_t) \neq 0,\quad \text{when } g(\xh_t)=0.
\ee

\begin{corollary}
\label{CorJump}
Assume that the state constraint is of first order.
Then if the control has a jump at time $\tau \in (0,T)$ for which $g(\xh_\tau)=0,$ then $\mu$ is continuous at $\tau$ for any associated multiplier $(\beta,\Psi,p,\dd \mu) \in \Lambda.$
\end{corollary} 

\begin{proof}
From the identity \eqref{uHu} in Lemma \ref{uhuk.l1}, if $[\uh_\tau]\neq 0,$ then $[H_u(t)] = 0.$ The latter implies that $[\mu_\tau]=0,$ in view of the second equality in \eqref{mudotg} and since the state constraint is of first order, i.e. \eqref{1order} holds.
\end{proof}

We refer to Remark \ref{RemJump2} regarding the link between previous Corollary \ref{CorJump} and Maurer \cite{Maurer77}.

\begin{remark}
\label{RemJump}
Let us illustrate by means of this example that the associated $\mu$ might have jumps at the initial and or at final times. In fact, consider the problem 
\benl
\begin{split}
\min\,\,& x_{1,T}+x_{2,T}; \\
 &\dot x_{1,t}=u_t\in [-1,1];\quad   \dot x_{2,t} = x_{1,t};\\
&g(x_{1,t})=-x_{1,t} \leq 0; \quad x_{1,0} =1, \; x_{2,0} =0.
\end{split}
\eenl
with $T=2,$ and note that
\benl
\uh_t :=
\left\{
\ba{rl}
-1 & \text{ on } [0,1],\\
0 & \text{ on } (1,2]
\ea
\right.
\eenl
is optimal for it.
Over $[0,1)$ the state constraint is not active, and on $[1,2]$ it is.
The Hamiltonian here is 
$H = p_1u + p_2 x_1,$ the costate being $p=(p_1,p_2).$
We obtain
$\dot p_2=0$, so that $p_{2,t}\equiv p_{2,T}=1$,  and
$-\dd p_1 = \dd t - \dd \mu$ with boundary condition $p_1(1)=1.$ Furthermore, on the constraint arc $(1,2)$ we necessarily have $H_u=0,$ leading to $p_1=0.$
We conclude that  $[p_T] = 1$ and therefore, due to \eqref{stateconstraint4}, $[\mu_T]=1,$ this is, $\mu$ jumps at the final time.
\end{remark}

\subsection{Arcs and junction points}\label{SecArcs}
The {\em contact} set associated with the state constraint is defined as
\be
C := \{ t\in [0,T]:\; g(\xh_t)=0 \}.
\ee
For $0 \leq a < b \leq T$, we say that $(a,b)$ is a 
{\em (maximal) active arc} for the state constraint,
or a {\em $C$ arc,} 
if $C$ contains $(a,b)$, but no open interval in which 
$(a,b)$ is strictly contained. Note that, since $t \mapsto g(\xh_t)$ is continuous, 
the set $C$ consists of a countable union of arcs, which can be ordered by size.
We say that 
$\tau\in (0,T)$ is a {\em junction point of the state constraint} 
if it is the extreme point of a $C$ arc.

We give similar definitions for the control constraint, paying
attention to the fact that the control variable is not
continuous, and is defined only almost everywhere.
So, we define the contact and
interior sets for the control bounds as
\be
\left\{ 
\begin{split}
&B_- := \{ t\in [0,T]:\; \uh_t = \umin  \},
\\
&B_+ := \{ t\in [0,T]:\; \uh_t = \umax  \},
\end{split}
\right.
\ee
and set $B:=B_- \cup B_+.$
We shall make clear that if $\umin = - \infty$ then $B_- = \emptyset$ and, analogously, if $\umax=+\infty$ then $B_+ = \emptyset.$
These sets are defined up to a null measure set and they can be identified with their characteristic functions.
We define arcs in a similar way, using representatives 
of the characteristic functions. That is, we say that
$(a,b)$ is a {\em $B_-$, $B_+$ arc} (or simply $B$ arc if we do not want to precise in which bound $\uh$ lies) if $(a,b)$  is included, up to a null measure set, in  $B_-$, $B_+,$ (in $B_-$ or $B_+$) respectively,
but no open interval strictly containing $(a,b)$ is.
Finally, let $S$ denote the {\em singular set} 
\be
S:=\{ t\in [0,T]: u_{\min} < \uh_t < u_{\max} \text{ and } g(\xh_t) <0 \}.
\ee
{\em Junction times} are in general points $\tau \in (0,T)$ at which the trajectory $(\xh,\uh)$ switches from one type of arc ($B_-,B_+,C$ or $S$) to another type. We may have, for instance, {\em $CS$ junctions, $BC$ junctions, etc. } With {\em $BC$ control} we refer to  a control that is a concatenation of a bang arc and a constrained arc, and we extend this notation for any finite sequence, i.e. $SC,$ $SBC,$ etc.

\begin{remark}
\label{RemJump2}
The result in Corollary \ref{CorJump} above was proved by Maurer
\cite{Maurer77} at times $\tau \in (0,T)$ being junction points of the
state constraint, and 
extended to state constraints of any order.
\end{remark}

Consider the following {\em geometric hypotheses} on the control structure:
\be
\label{chypradn}
\left\{ \ba{cl}
 {\rm (i)} &\text{the interval $[0,T]$ is (up to a zero measure set)
the disjoint}
\\ & \text{union of finitely   many arcs of type $B$, $C$ and $S,$}
\\
{\rm (ii)} &\text{the control $\uh$ is at uniformly positive distance
  of the bounds } 
\\& u_{\min} \text{ and } u_{\max}, \text{ over $C$ and $S$ arcs,}
\\ 
{\rm (iii)} &
\text{the control $\uh$ is discontinuous at CS and SC junctions.}
\ea \right.
\ee
\begin{remark}
Condition \eqref{chypradn}(ii) implies the discontinuity of the
control at
 $BS$ and $SB$ junctions.
Therefore, \eqref{chypradn} yields the discontinuity of the
control at any junction. 
\end{remark}

Observe that the example in Remark \ref{RemJump} above 
verifies the structural hypothesis \eqref{chypradn}.

Let us note that we can rewrite the condition \eqref{1order} of first order state constraint as
\be
\label{1order'}
g'(\xh_t) f_1(\xh_t) \neq 0,\quad \text{on } C.
\ee
From $g(\xh_t)=0$ on $C,$ we get
\be
 \label{dtg0}
 0 = \ddt g(\xh_t) = g'(\xh_t) \big( f_0(\xh_t) + \uh_t f_1(\xh_t) \big),\quad \text{on } C,
 \ee
 and, whenever \eqref{1order'} holds,   we obtain that 
\be
 \label{uinC}
 \uh_t=  -\frac{g'(\xh_t)f_0(\xh_t)}{g'(\xh_t)f_1(\xh_t)},\quad \text{on } C.
 \ee
 
 {\em In the remainder of the article we assume that $(\xh,\uh)$ satisfies the geometric hypotheses \eqref{chypradn} and that the state constraint is of first order, that is,  \eqref{1order'} holds true.}

\subsection{About $\dd \mu$}\label{diffopcosec}

Observe that, along a constrained arc $(a,b)$, over which 
$u_{\min} < \uh_t < u_{\max}$ a.e., in view of the minimum condition
\eqref{lmscac4a}, we have that 
$H_u=p f_1=0,$ for any $(\beta,\Psi,p,\dd\mu) \in \Lambda.$ Differentiating this equation with respect to time and using \eqref{costateeq}, we get
\be
\label{dHu0}
\begin{split}
0 & = \dd H_u = 
p[f_1,f_0] \dd t- g' f_1\dd\mu.
\end{split}
\ee
Thus, since the state constraint is of first order, $\dd\mu$ has a density $\nu\geq 0$ over $C$ given by the absolutely continuous function
\be
\label{nu}
\nu:[0,T] \to \cR,\quad
t \mapsto \nu_t := \frac{p_t[f_1,f_0](\xh_t)}{g'(\xh_t)f_1(\xh_t)},
\ee
where $[X,Y]:= X'  Y - Y'
X$ denotes the {\em Lie bracket} associated to a pair of vector fields $X,Y : \cR^n \to \cR^n.$

\section{Second order necessary conditions}\label{SecSOC}

In this section we state second order necessary conditions for weak optimality of problem \eqref{P}. We start by defining the cones of critical directions and giving necessary conditions, obtained by applying the abstract result in Theorem \ref{appt2} in the Appendix. Afterwards, we present the Goh transformation \cite{Goh66} and we state second order conditions in the transformed variables.

\subsection{Critical directions}
 Let us extend the use of  $z[v,z^0]$ to
denote the solution of the linearized state equation \eqref{lineq}-\eqref{lineq0} for $(v,z^0) \in L^2 \times\cR^n ,$ and let us write $T_\Phi$  to refer to the {\em tangent cone} to $\{ 0 \}_{\RR^{n_1}} \times \RR^{n_2}_-$ at the point $(\xh_0,\xh_T)$,
 given by 
\benl
T_\Phi :=  \{0\}_{\RR^{n_1}} \times \{ \eta\in\RR^{n_2}; \; 
\eta_i\leq 0 : \text{if $\Phi_{n_1+i}(\xh_0,\xh_T)=0$,
$i=1,\ldots,n_2$} \}.
\eenl
For $v \in L^2$ and $z_0 \in \cR^n,$ consider the {\em linearization of the cost and endpoint constraints}
\be
\label{linPhi0}
\left\{
\ba{l}
\phi'(\xh_0,\xh_T)(z_0,z_T) \leq0,\\
\Phi'(\xh_0,\xh_T)(z_0,z_T) \in T_\Phi,
\ea 
\right.
\ee
where $z:=z[v,z_0].$
Define the {\em critical cone} as the set
\be
\C := 
\left\{ 
\begin{split}
&(v,z)\in L^\infty \times (W^{1,\infty})^n : (v,z) \text{ satisfies \eqref{lineq} and \eqref{linPhi0}},\\
&v_t\geq 0 \,\text{ a.e. on } B_-,\, v_t\leq0\, \text{ a.e. on } B_+,\quad g'(\xh_t) z_t \leq 0\ \text{ on } C
\end{split}
\right\}. 
\ee      
Define the {\em strict critical cone,} 
 \be
 \label{CS}
\C_S := 
\left\{ 
(v,z)\in \C: 
v_t = 0 \, \text{ a.e. on } B_\pm,\ g'(\xh_t) z_t =0\ \text{ on } C
\right\}. 
\ee    
Note that the strict critical cone is a polyhedron of a closed subspace of $L^\infty\times (W^{1,\infty})^n.$  
Consider 
the following  {\em weak complementarity condition}:  
there exists a Lagrange multiplier $(\beta,\Psi,p ,\dd\mu) \in
\Lambda$ (defined in \eqref{Lambda}) such that 
\be
\label{weakcomp}
\left\{
\ba{l}
\text{(i)}\ H_{u}(\uh,\xh,p) \neq 0,\quad \text{a.e.  over }\,B, \\
\text{(ii) the support of $\dd\mu$ is $C.$}
\ea
\right.
\ee 
We have the following identity:

\begin{proposition}
Assume that the weak complementarity condition \eqref{weakcomp} holds, then the critical and strict critical cones coincide, this is $\C = \C_S.$
\end{proposition}

\subsection{Radiality of critical directions}
\label{radialtc}

In order to prove the necessary condition of Theorem \ref{SONCQ} below, we make use of the abstract result in Theorem \ref{appt2} of the Appendix and of its related concepts. We consider in particular {\em radial critical directions} as given in Definition \ref{tool_ocso.d2} of
Section \ref{tool_ocso}. For our optimal control problem, an element $(v,z)$ of  $ \C$ is said to be {\em radial} 
if and only if, for some $\sigma>0,$ the following conditions are satisfied:
\be
\label{prrad.p_1}
\umin \leq \uh_t +\sigma v_t \leq \umax, \quad \text{a.e. on } [0,T],
\ee
\be
\label{prrad.p_2}
g(\xh_t)+\sigma g'(\xh_t) z_t \leq 0,\quad \text{for all } t\in [0,T].
\ee

Recall hypotheses \eqref{chypradn} and \eqref{1order'} on the control structure and the order of the state constraint, respectively.
\begin{proposition} 
\label{PropRadial}
Any critical direction (in $\C$) is radial. 
\end{proposition}

\begin{proof} 
Let $(v,z)\in \C$.
Relation \eqref{prrad.p_1} holds for $\sigma>0$ small enough over $B$ arcs, and on $S$ and $C$ arcs by  \eqref{chypradn}(ii).
Relation \eqref{prrad.p_2} trivially holds over $C$ arcs and, since
the state constraint is not active over $B$ and $S$ arcs, it does also
hold over these arcs, except perhaps in the vicinity of entry or exit points to $C$ arcs.  
For $t\in [0,T]$, let $\delta_t$ denote the distance between
$t$ and $C.$  
By \eqref{chypradn}(ii)-(iii) and \eqref{1order'},
$\ddt g(\xh_t)$ has a jump at the entry and exit points of any $C$ arc.
Let us check that, for $\eps>0$ and 
$\sigma>0$ small enough, we have 
$g(\xh_t)  + \sigma g'(\xh_t) z_t \leq 0$ for all 
$t\in [0,T]$ such that $\delta_t \in (0, \eps).$ 
For such $t$, 
reducing $\eps>0$ if necessary,  we have that 
$g(\xh_t) \leq - c_1 \delta_t$.
On the other hand, since 
$g'(\xh_t) z_t$ is Lipschitz continuous and nonpositive over C arcs,
we have that
$g'(\xh_t) z_t \leq c_2 \delta_t$, 
where $c_2>0$ depends on $v.$
Therefore,
$g(\xh_t)  + \sigma g'(\xh_t) z_t \leq (\sigma c_2-c_1) \delta_t< 0$
as soon as $\sigma < c_1/c_2$.
The conclusion follows.
\end{proof}

We next give an example of a problem in which the  optimal control
is not discontinuous at the junction points of C arcs,
but whose associated critical cone 
is nevertheless radial since it reduces to $\{0\}.$

\begin{remark}
\label{CBcont}
Let us consider the problem
\be
\begin{split}
& \min \int_0^2 x_{1,t}\dd t, \\
& \dot{x}_{1,t} =u_t \in [-1,1],
\;\; 
 \dot{x}_{2,t} =1, \; \text{a.e. on $[0,T],$}
\\
& x_{1,0} = 1,\, x_{2,0} =0; \;\;
-(x_{2,t}-1)^2 - x_{1,t} \leq 0 \; \text{over $[0,T]$}.
\end{split}
\ee
Notice that $x_{2,t} = t$ and that the state
constraint is of first order since
$g(x) := -(x_2-1)^2 - x_1$ satisfies 
$
\ddt g(x_t) = -2(x_{2,t}-1) -u+t.
$
Thus,  $u_t = -2(t-1)$ on a constraint arc.
It is easy to see that the optimal control is of type $B_-CB_-.$ More precisely, 
\be
\uh_t =
\left\{
\ba{cl}
-1, & \quad t\in[0,1], \\
-2(t-1), & \quad t\in(1,3/2],\\
-1,  & \quad t\in (3/2,2]. \
\ea
\right.
\ee
Thus, $\uh$ is continuous at the junction time $t=3/2.$ 
Yet, since no singular arc occurs, the critical cone reduces to $\{0\}$ and, therefore, 
 any critical direction is radial.
\end{remark}

\subsection{Statement of second order necessary conditions} \label{lmsonc-h}
Next we state a second order necessary condition in terms of the Hessian of the Lagrangian $\LL$ (which respect to $(u,x)$), which is given by the quadratic form
\be
\label{Q}
Q := Q^0+ Q^E  + Q^g,
\ee
where 
\be
\label{quad}
\left\{ \begin{split}
Q^0(v,z,\lambda) &:= \int_0^T \big( z_t^\top\, H_{xx}\,z_t + 2 v_t H_{ux}\,z_t \big)
\dd t,
\\
Q^E(v,z,\lambda)&:= D^2 \ell^{\beta,\Psi}(z_0,z_T)^2,
\\
Q^g(v,z,\lambda) &:= \int_0^T z_t^\top\, g''\,z_t  \dd \mu_t,
\end{split}
 \right.
\ee
We recall that the Lagrangian  $\LL$ was defined in \eqref{lmscac5}.

\begin{proposition} For every multiplier $\lambda \in \Lambda,$ we have
\be
D^2_{(u,x)^2} \LL(\uh,\xh,\lambda)(v,z)^2 = Q(v,z,\lambda),\quad \text{ for all } (v,z) \in \C.
\ee
\end{proposition} 

As a consequence of Proposition \ref{PropRadial} and Theorem \ref{appt2} in the Appendix,
the following result holds.
\begin{theorem}[Second order necessary condition]
\label{SONCQ}
Assume that $(\uh,\xh)$ is a weak minimum of problem \eqref{P}.
Then 
\be
\max_{\lambda \in \Lambda} Q(v,z,\lambda) \geq 0,\quad \text{ for all } (v,z) \in \C.
\ee
\end{theorem}

\subsection{Goh transformation and primitives of critical directions} \label{lmpcd}
        
For $(v,z^0) \in L^\infty \times \cR^n,$ and $z:=z[v,z^0]$ being the solution of the linearized state equation \eqref{lineq}-\eqref{lineq0}, let us set, 
\be
\label{Goh}
y_t := \int_0^t v_s\dd s; \quad 
\xi_t := z_t - y_t f_1(\xh_t),\quad \text{ for } t\in[0,T]. 
\ee
This change of variables is called  {\em Goh transformation} \cite{Goh66}. Defining $E_t:= A_t f_1(\xh_t)- \ddt f_1(\xh_t)$ (where $A$ was defined in \eqref{AB}),
observe that $\xi$ is solution of 
\be
\label{eqxitab}
\dot \xi_t = A_t \xi_t + y_t E_t, 
\ee
on the interval $[0,T],$ with initial condition
$$ \xi_0 = z^0.$$
 
 Consider the set of {\em strict primitive critical directions}  
$$
\P_S := 
\left\{
(y,h,\xi)\in  W^{1,\infty}\times \cR \times
(W^{1,\infty})^n  : y_0=0,\
y_T=h,\ (\dot y, \xi + y f_1 )\in \C_S
\right\},
$$
and let $\overline{\P_S}$ denote its closure with respect to the $L^2 \times \cR \times (H^1)^n-$topology.
The final value $y_T$ of $y$ is involved in the definition of $\P_S$ since it becomes an independent variable when we consider its closure.
We provide a characterization of $\overline{\P}_S$ in Theorem \ref{CharP2} below.

Let us note that if $(v,z) \in \C_S$ is a strict critical direction, then $(y,\xi)$ 
given by Goh transform \eqref{Goh} satisfies the following conditions:
\be
\label{intcond}
\left\{ \ba{ll}
 {\rm (i)} &
 g'(\xh_t) ( \xi_t + y_t f_1(\xh_t) ) =0\ \text{on } C,
\\ {\rm (ii)} &
\text{$y$ is constant on each $B$ arc,} 
\ea\right.\ee 
and if $(v,z)$ satisfies the linearized endpoint 
relations \eqref{linPhi0},
then 
\be
\label{transend}
\left\{
\begin{split}
& \phi'(\xh_0,\xh_T) \big( \xi_0,\xi_T+h f_{1}(\xh_T) \big) \leq 0,\\
& \Phi'(\xh_0,\xh_T)\big( \xi_0,\xi_T+h f_{1}(\xh_T) \big) \in T_\Phi,
\end{split}
\right.
\ee
where we set $h:= y_T.$
In the sequel we let 
 $0=: \hat\tau_0 < \hat\tau_1 < \dots < \hat\tau_N :=T$ 
denote the union of the set of junction times with $\{0,T\}.$ 

A $B$ arc starting at time 0 
(respectively ending at time $T$) is called 
a $B_{0\pm}$ (respectively $B_{T\pm}$) {\em arc.}

\begin{proposition}
\label{LimCond}
Any $(y,h,\xi) \in \overline{\P_S}$
  verifies \eqref{intcond}-\eqref{transend} and
\be
\label{limitcond}
\left\{
\ba{ll}
 {\rm (i)} & \text{ $y$ is continuous at the $BC,$ $CB$ and $BB$ junctions,} \\
 {\rm (ii)} & \text{ $y_t=0,$ on $B_{0\pm}$ if a $B_{0\pm}$ arc exists,}\\
 {\rm (iii)} & \text{ $y_t=h,$ on $B_{T\pm}$ if a $B_{T\pm}$ arc exists,}\\
 {\rm (iv)} & \lim_{t \uparrow T} y_t = h,\ \text{if } T\in C.
 \ea
 \right.
 \ee
 \end{proposition}
 
 \begin{proof}
Let $(y,h,\xi) \in\overline{\P_S}$ be the limit of a sequence 
$(y^k,y^k_T,\xi^k)_k \subset \P_S$ in the
$L^2\times \cR \times (H^1)^n-$topology. 
By the Ascoli-Arzel\`a theorem, if $(y^k)$ is 
equi-bounded and equi-Lipschitz continuous over an interval $[a,b],$   then $y^k\rar y$ uniformly over $[a,b]$ and the limit function $y$ is Lipschitz continuous on $[a,b].$
Consequently, we obtain that
$y$ is constant over B
arcs, and continuous at $BC,$ $CB$ and $BB$ junctions.
We can prove the other statements by an analogous reasoning. 
 \end{proof}

Define the set
\be
\label{PS2}
\P^2_S := 
\left\{
(y,h,\xi) \in L^2 \times \cR \times (H^1)^n:\,\text{ \eqref{eqxitab}-\eqref{limitcond} hold} \right\}.
\ee
Then the following characterization holds.

\begin{theorem}
\label{CharP2}
We have that  $\P^2_S=\overline{\P_S}$.
 \end{theorem}

\begin{proof}
That $\overline{\P_S}  \subset \P^2_S$ follows from Proposition
\ref{LimCond}.
Let us prove the converse inclusion. 
Define  the linear space
\be
\Z :=\left\{
\begin{split}
& (y,y_T,\xi)\in    W^{1,\infty}  \times \cR \times  
\left( \prod_{j=1}^N W^{1,\infty}(\hat\tau_{j-1},\hat\tau_j)^n\right)\\
&y_0=0,\ \text{\eqref{eqxitab} holds at each }
(\hat\tau_{j-1},\hat\tau_j),\,\text{and \eqref{intcond}  holds} 
\end{split}
\right\},
\ee
that is obtained from $\P_S$ by removing 
condition \eqref{transend} and allowing
$\xi$  to be discontinuous
at the junctions $\hat\tau_j$, $j=1,\ldots N-1$. 

Let $(y,h,\xi) \in \P^2_S.$ For any $\eps>0,$ 
we now construct $(y_\eps,y_{\eps,T},\xi_\eps) \in \Z $ such that $y_{\eps,T}=h$ and
\be
\label{esteps}
\|y-y_\eps\|_2+\|\xi-\xi_\eps\|_\infty = o(1).
\ee
First set 
\be
y_{\eps,t} := y_t,\quad 
\xi_{\eps,t} := \xi_t,\quad 
\text{on } B\cup C.
\ee
On $S,$ define $y_\eps$ in such a way that  $y_\eps\in W^{1,\infty}(0,T),$ the values 
at the junction times are fixed in the following way:
\be
\label{yepsjunction}
\left\{
\begin{split}
& y_\eps(\hat\tau_j+) = y(\hat\tau_j-),\quad \text{if } \hat\tau_j>0 \text{ is an entry point of an $S$ arc,}\\
& y_\eps(\hat\tau_j-) = y(\hat\tau_j+),\quad \text{if } \hat\tau_j < T \text{ is an exit point of an $S$ arc,}\\
& y_{\eps,0} = 0, \quad \text{if }0\in S;\quad
  y_{\eps,T} = h, \quad \text{if }T\in S,
\end{split}
\right.
\ee
and $\|y-y_\eps\|_2 < \eps.$
Such an $y_\eps$ exists, see \cite[Lemma 8.1]{ABDL11}. 
Define $\xi_\eps$ over each $S$ arc
by integrating \eqref{eqxitab} over the respective arc with $y=y_\eps$ and 
 the initial condition 
$\xi_{\eps,\tau} = \xi_{\tau}$,
where $\tau$ denotes the entry point of the arc.
Then  $(y_\eps,y_{\eps,T},\xi_\eps) \in \Z$ 
satisfies the estimate \eqref{esteps}.
In particular, we have
\begin{gather}
\label{xijump}
|\xi_\eps(\hat\tau_j-)-\xi_\eps(\hat\tau_j+)|=
|\xi_\eps(\hat\tau_j-)-\xi(\hat\tau_j)|=
o(1),\quad \text{for all } j=1,\dots,N-1,\\
\label{endjump1}
|\phi'(\xh_0,\xh_T)(\xi_{\eps,0},\xi_{\eps,T}+f_1(\xh_T)y_{\eps,T}) -
\phi'(\xh_0,\xh_T)(\xi_{0},\xi_{T}+f_1(\xh_T) h) | = o(1),
\\
\label{endjump2}
|\Phi'(\xh_0,\xh_T)(\xi_{\eps,0},\xi_{\eps,T}+f_1(\xh_T)y_{\eps,T}) -
\Phi'(\xh_0,\xh_T)(\xi_{0},\xi_{T}+f_1(\xh_T) h) | = o(1).
\end{gather}

Notice that the cone $\P_S$ is obtained from $\Z$ by adding the constraints in \eqref{transend} and the continuity conditions
\be
\xi(\hat\tau_j-)-\xi(\hat\tau_j+)=0,\quad \text{for all } j=1,\dots,N-1.
\ee

 In view of Hoffman's lemma \cite{Hof52} and estimates \eqref{xijump}-\eqref{endjump2},
 we get that there exists $(\tilde y_\eps, \tilde y_{\eps,T}, \tilde\xi_\eps) \in \P_S$ such that
\be
\label{esteps2}
\|\tilde y_\eps-y_\eps\|_2+\|\tilde \xi_\eps-\xi_\eps\|_\infty = o(1).
\ee
Finally, from  \eqref{esteps} and \eqref{esteps2} we have that
\be
\|\tilde y_\eps-y\|_2+\|\tilde \xi_\eps-\xi\|_\infty = o(1),
\ee
and hence, the density of $\P_S$ in $\P_S^2$ follows.
\end{proof}

\subsection{Goh transformation on the Hessian of Lagrangian} \label{lmsonc-hh}

Next, we want to express each of the quadratic functions
in \eqref{Q}-\eqref{quad} as functions of $(y,h,\xi)$.
For the terms that are quadratic in $z$, it suffices to replace $z$
by $\xi+f_1(\xh) y$ as given in Goh transform. With this aim, set for $(y,h,\xi) \in L^2 \times \cR \times (H^1)^n$ and $\lambda := (\beta,\Psi,p,\dd\mu) \in \Lambda,$
\be
\begin{split}
\Om_T(y,h,\xi,\lambda) :=  &
2 h\, H_{ux,T}   \xi_T +  h\, H_{ux,T}  f_1(\xh_T) h,
\\
\Om^0(y,h,\xi,\lambda) :=  &
\int_0^T \big( \xi_t^\top H_{xx}\xi_t + 2y_t M \xi_t + y_t R  y_t
\big) \dd t,
\\
\Om^E (y,h,\xi,\lambda) :=
\,&\,D^2\ell^{\beta,\Psi} (\xi_0,\xi_T+f_1(\xh_T)h)^2,
\\
\Om^g(y,h,\xi,\lambda) :=&\, 
\int_0^T (\xi_t + f_{1}(\xh_t)y)^\top g''(\xh_t)(\xi_t + f_{1}(\xh_t)y)  \dd \mu_t,\\
\Om:=\,& \, \Om_T + \Om^0 + \Om^E + \Om^g,
\end{split}
\ee
with 
\be
\label{M}
M:= f_1^\top H_{xx} - \dot{H}_{ux} - H_{ux}A,
\ee
 \be
 \label{R}
 R:= f_1^\top H_{xx} f_1-2H_{ux}E- \frac{\dd}{\dd t}(H_{ux}f_1).
 \ee
 Here, when we omit the argument of $M,$ $R,$  $H$ or its derivatives, we mean that they are evaluated at $(\xh,\uh,\lambda).$ 
\begin{remark}
Easy computations show that $R$ does not depend on $u.$ More precisely, $R$ is given by
\be
R(\xh_t,\lambda_t) = p_t \big[ [f_0,f_1],f_1 \big] (\xh_t) + g'(\xh_t) f_1'(\xh_t)f_1(\xh_t) \nu_t,
\ee
where $\nu$ is the density of $\dd\mu$ given in \eqref{nu}.
\end{remark}

\begin{proposition}
\label{QOm}
Let $(v,z) \in L^2 \times (H^1)^n$ be a  solution of \eqref{lineq} and
let $(y,\xi)$ 
be defined by the Goh transformation \eqref{Goh}. Then, for any $\lambda \in \Lambda,$ 
\be
Q(v,z,\lambda) = \Om(y,y_T,\xi,\lambda).
\ee
\end{proposition}

\begin{proof}
Take $(v,z)$ and $(y,\xi)$ as in the statement. It is straightforward to prove that 
\be
Q^E(v,z,\lambda)=\Om^E(y,y_T,\xi,\lambda),\ \text{ and }\ 
Q^g(v,z,\lambda)=\Om^g(y,y_T,\xi,\lambda).
\ee
In order to prove the equality between $Q^0$ and $\Om^0,$ let us replace each occurrence of $z$ in $Q^0,$ by its expression in the Goh transformation \eqref{Goh}, i.e. change $z$ by $\xi+f_1(\xh)y.$ The first term in $Q^0$ can be written as
\be
\label{Q00}
\int_0^T z^\top H_{xx} z \dd t
= \int_0^T (\xi+f_1y)^\top H_{xx} (\xi+f_1y) \dd t.
\ee
Let us consider the second term in $Q^0:$
\be
\label{Q0}
\begin{split}
\int_0^T  v H_{ux} (\xi+f_1y)  \dd t = \int_0^T \big( v H_{ux} \xi + v H_{ux} f_1 y \big) \dd t.
\end{split}
\ee
Integrating by parts the first term in previous equation we get
\be
\label{Q01}
\begin{split}
\int_0^T v H_{ux} \xi \dd t =&\,
\left[ y H_{ux} \xi \right]_0^T - \int_0^T y \big( \dot{H}_{ux} \xi + H_{ux} \dot{\xi} \big) \dd t \\
=&\, y_T H_{ux,T} \xi_T - \int_0^T  y \big( \dot{H}_{ux} \xi +H_{ux}A\xi + H_{ux} Ey)\dd t.
\end{split}
\ee
For the second term in the right hand-side of \eqref{Q0} we obtain
\be
\begin{split}
\int_0^T v H_{ux} f_1 y \dd t = &\, 
\left[ y H_{ux} f_1 y\right]_0^T - \int_0^T y \left( \frac{\dd}{\dd t}(H_{ux}f_1) y+ H_{ux}f_1 v \right)\dd t.
\end{split}
\ee
This identity yields the following
\be
\label{Q02}
\int_0^T v H_{ux} f_1 y \dd t = \half y_T H_{ux,T} f_1 y_T
-\half \int_0^T y^2 \frac{\dd}{\dd t}(H_{ux}f_1)  \, \dd t.
\ee
From \eqref{Q00}, \eqref{Q01} and \eqref{Q02} we get the desired result.
\end{proof}

\subsection{Second order necessary condition in the new variables}

We can obtain the following new necessary condition in the variables after Goh's transformation.

\begin{theorem}
\label{SONCOm}
If  $(\uh,\xh)$ is a weak minimum, then
\be
\max_{\lambda \in \Lambda} \Om (y,h,\xi,\lambda) \geq 0,\quad \text{for all } (y,h,\xi) \in \P^2_S.
\ee
\end{theorem}

\begin{proof}
Let us assume first that the qualification condition \eqref{tool_ocfo.qualifoc}(i) does not hold. Therefore,  there exists
a nonzero element $\Psi_E$ in $ [{\rm Im }\, D \bar\Phi_E(\uh,\xh_0)]^\perp.$ 
Hence, the multiplier $\lambda$ composed by such  $\Psi_E$, and having $(\beta,\Psi_I,\dd \mu):=0$ and the associated costate $p,$ is a Lagrange multiplier, as well as its opposite $-\lambda.$ Therefore, either $ \Om(y,h,\xi,\lambda)$ or $ \Om (y,h,\xi,-\lambda)$ is greater or equal zero, and the conclusion follows.

Let us now consider the case when \eqref{tool_ocfo.qualifoc}(i) holds true.
Then, the corresponding abstract problem \eqref{Pabs} (defined in the Appendix \ref{tool_ocfo}) verifies the qualification condition \eqref{qualif}(i). 
Recall the definition of $\hat\Lambda$ in \eqref{hatLambda}, and the corresponding set $\hat\Lambda_1$ (see \eqref{LambdaBeta}).
In view of Theorem \ref{tool_ocfo.t2}, $\hat\Lambda_1$ is non empty
and bounded. Furthermore, due to the
Banach-Alaoglu Theorem, and since the space of continuous functions in $[0,T]$ is separable, we get that $\hat\Lambda_1$ is weakly* sequentially compact. 

Consider now $(y,h,\xi) \in \P^2_S.$ 
By Theorem \ref{CharP2}, there exists a sequence $(y^k,y_T^k,\xi^k)$
in $\P_S$ such that
\be
(y^k,y_T^k,\xi^k) \to (y,h,\xi),\quad \text{in the } L^2\times \cR \times (H^1)^n\text{-topology}.
\ee
By Proposition \ref{QOm}, for all $\lambda\in\Lambda,$
\be
\Om(y^k,y_T^k,\xi^k,\lambda) = Q(v^k,z^k,\lambda),
\ee
where $v^k := \frac{\dd}{\dd t} y^k$ and $z^k := \xi^k + f_1 y^k.$ For each $(v^k,z^k),$ due to Theorem \ref{SONCQ}, there exists $\lambda^k \in \Lambda$ for which
\be
\label{QkNC}
Q(v^k,z^k,\lambda^k)\geq 0.
\ee
Let $\hat\lambda_k$ be the corresponding element of $\hat\Lambda_1$ given by the bijection  \eqref{maplambda}, and consider $\bar\lambda_k \in \Lambda$ such that $\hat\lambda_k = (1,\bar\lambda_k).$
Then
\be
Q(v^k,z^k,\bar\lambda^k)\geq 0,
\ee
in view of \eqref{QkNC} and since $\bar\lambda_k$ is obtained from $\lambda_k$ by multiplying by a positive scalar.
Given that $\hat\Lambda_1$ is weakly* sequentially compact, there exists a subsequence 
$(\hat\lambda^{k_j})_j$ weakly* convergent  to $\hat\lambda=(1,\bar\lambda) \in \hat\Lambda_1,$ where $\bar\lambda\in \Lambda.$
Thus, $\bar\lambda_k \overset{*}{\rightharpoonup} \bar\lambda$ in $\Lambda$ and we get
\be
 \Om(y,h,\xi,\bar\lambda) = \lim_{k \to \infty} \Om(y^k,y_T^k,\xi^k,\bar\lambda^k) = \lim_{k \to \infty}  Q(v^k,z^k,\bar\lambda^k)\geq 0.
\ee
This concludes the proof.

\end{proof}

\section{Second order sufficient conditions}\label{SecSC}

In this section we show a second order sufficient condition in terms of the uniform positivity of $\Omega$ and guaranteeing that the nominal solution $(\uh,\xh)$ is a strict Pontryagin minimum whenever this condition is satisfied.

To state the main result of this section (see Theorem \ref{SC}) we need to introduce the following concepts.

\begin{definition}
\label{defminpont}
We say that $(\uh,\xh)$ is a {\em Pontryagin minimum} of problem \eqref{P} if for any $M>0,$ there exists $\eps_M>0$ such that $(\uh,\xh)$ is a minimum in the set of feasible trajectories $(u,x)$ satisfying
\be
\label{defminpont1}
\|x-\xh\|_\infty +\|u-\uh \|_1 < \eps_M,\quad \|u-\uh\|_\infty < M.
\ee
A sequence $(v_k) \subset L^\infty$ is said to {\em converge to 0 in the Pontryagin sense} if $\|v_k\|_1 \to 0$ and there exists $M>0$ such that $\|v_k\|_\infty \leq M,$ for all $k\in \cN.$
\end{definition}

\begin{definition}
Let us define the function $\gamma:L^2 \times \cR \times \cR^n \to \cR,$ given by
\be
\gamma(y,h,x_0) := \int_0^T y_t^2 \dd t + h^2 + |x_0|^2.
\ee
We say that $(\uh,\xh)$ satisfies the {\em $\gamma-$growth condition in the Pontryagin sense} if there exists $\rho>0$ such that, for every sequence of feasible variations $(v_k,\delta x_k) $ having $(v_k)$ convergent to 0 in the Pontryagin sense and $\delta x_{k,0} \to 0,$ one has
\be
\label{gg} 
\phi(\xh_0+\delta x_{k,T}\,,\xh_T+\delta x_{k,T}) - \phi(\xh_0,\xh_T) \geq \rho \gamma(y_k,y_{k,T},\delta x_{k,0}),
\ee
for $k$ large enough, where $y_{k,t}:=\int_0^tv_{k,s} {\rm d}s,$ for $t\in [0,T].$
 \end{definition}

\begin{definition}\label{sc}
We say that $(\uh,\xh)$ satisfies {\em strict complementarity condition for the control constraints} if the following conditions hold:
\begin{itemize}
\item[(i)]
$\ds\max_{\lambda \in \Lambda} H_u(\uh_t,\xh_t,p_t) > 0,$ for all $t$ in the interior of $B_-,$ 
\item[(ii)]
$\ds\min_{\lambda \in \Lambda} H_u(\uh_t,\xh_t,p_t) < 0,$ for all $t$ in the interior of $B_+,$
\item[(iii)]
$\ds\max_{\lambda \in \Lambda} H_u (\uh_0,\xh_0,p_0) > 0,$  if  $0\in B_- ,$ and
$\ds\max_{\lambda \in \Lambda} H_u (\uh_T,\xh_T,p_T) > 0,$ if $T\in B_-,$ \\
$\ds\min_{\lambda \in \Lambda} H_u (\uh_0,\xh_0,p_0) < 0,$  if  $0\in B_+ ,$ and
$\ds\min_{\lambda \in \Lambda} H_u (\uh_T,\xh_T,p_T) < 0,$ if $T\in B_+.$
\end{itemize}
\end{definition}

Consider the following {\em extended cones of critical directions}
(compare to $\P_S^2$ defined in \eqref{PS2}):

\be
\ba{lll}
\label{hatP2}
\hat{\P}^2&:= 
\{(y,h,\xi)\in L^2 \times \cR \times (H^1)^n: \text{\eqref{eqxitab}-\eqref{transend},\ \eqref{limitcond}(ii)-(iii) hold}\},\\
\hat{\hat{\P}}^2&:= 
\{(y,h,\xi)\in \hat\P^2 : \text{\eqref{limitcond}(iv) holds}\},
\ea\ee
and 
\be
\label{tildeP2}
\P^2_*:=
\left\{
\ba{ll}
\hat{\hat{\P}}^2, & \text{if } T\in C \text{ and } [\mu(T)]>0, \text{ for some } (\beta,\Psi,p,{\rm d}\mu)\in \Lambda,\\
\hat{\P}^2, & \text{otherwise.}
\ea
\right.
\ee

Let us recall the definition of Legendre form (see e.g. \cite{Hes51}):
\begin{definition}
 Let $W$ be a Hilbert space. We say that a quadratic mapping $\mathcal{Q}:W\to \cR$ is a {\em Legendre form} if it is sequentially weakly lower semi continuous such that, if $w_k \to w$ weakly in $W$ and $\mathcal{Q}(w_k) \to \mathcal{Q}(w),$ then $w_k \to w$ strongly.
\end{definition}

\begin{theorem}
\label{SC} 
Suppose that the following conditions hold true:
\begin{itemize}
\item[(i)] $(\uh,\xh)$ satisfies strict complementarity for the control constraint and the complementarity for the state constraint \eqref{weakcomp}(ii);
\item[(ii)] for each $\lambda \in \Lambda,$ $\Omega(\cdot,\lambda)$ is a Legendre form in the linear space $\{(y,h,\xi)\in L^2 \times \cR \times (H^1)^n:\text{\eqref{eqxitab} holds}\};$ and 
\item[(iii)] there exists $\rho > 0$ such that
\be
\label{unifpos}
\max_{\lambda \in \Lambda} \Omega(y,h,\xi,\lambda) \geq \rho \gamma(y,h,\xi_0),\quad \text{for all } (y,h,\xi) \in \P^2_*.
\ee
\end{itemize}
Then $(\uh,\xh)$ is a Pontryagin minimum satisfying $\gamma-$growth.
\end{theorem}

\begin{remark}
We have that $\Om(\cdot,\lambda)$ 
is a Legendre form iff
 there exists $\alpha>0,$ such that,  for every $\lambda \in \Lambda,$
\be
R(\xh_t,\lambda_t)+ f_1(\xh_t)^\top g''(\xh_t) f_1(\xh_t) \nu_t > \alpha,\quad \text{on } [0,T],
\ee
where $\nu$ is the density of ${\rm d}\mu,$ which vanishes on $[0,T] \backslash C,$ and is given by the expression \eqref{nu} on the set $C.$ See \cite{Hes51} for more details.
\end{remark}

The remainder of this section is devoted to the proof of Theorem
\ref{SC}. We first show the technical results in Lemma
\ref{expansionL}, Propositions \ref{PropLexpansion} and
\ref{Propintg'}, and then give the proof of the Theorem \ref{SC}. 
\if{An example verifying the conditions of this theorem is provided
  later on in Section \ref{SecExamples}.
} \fi

For the lemma below recall the definition of the Lagrangian function $\LL$ which was given in \eqref{lmscac5}.
\begin{lemma}
\label{expansionL}
Let $(u,x) \in L^2 \times (H^1)^n$ be a solution of \eqref{bsbstateeq}, and set $(\delta x,\delta u):=(x,u)-(\xh,\uh).$ 
Then, the following expression for the Lagrangian function holds for every multiplier $\lambda \in \Lambda,$
\be
\LL(u,x,\lambda) = \LL(\uh,\xh,\lambda) + \tilde{Q}(\delta u,\delta x,\lambda) + \mathpzc{r}(\delta u,\delta x,\lambda),
\ee
where
\benl
\tilde{Q}(\delta u,\delta x,\lambda) := \int_0^T  H_u \delta u \dd t + Q(\delta u,\delta x,\lambda)
\eenl
and
\benl
\begin{split}
&\mathpzc{r}(\delta u,\delta x) := \mathcal{O} \big( |(\delta x_0,\delta x_T)|^3 \big)\\
&+ \int_0^T \big\{
\half H_{uxx}(\delta u,\delta x,\delta x) + p(\uh+\delta u)  \mathcal{O} \big( |\delta x|^3 \big) \big\} {\rm d} t+ \int_0^T  \mathcal{O} \big( |\delta x|^3 \big) {\rm d}\mu .
\end{split}
\eenl
\end{lemma}

\begin{proof}
Let us consider the following second order Taylor expansions, written in compact form,
\be
\label{ellfg}
\begin{split}
\ell^{\beta,\Psi} (x_0,x_T) &= \ell^{\beta,\Psi} + D\ell^{\beta,\Psi} (\delta x_0,\delta x_T) + \half D^2 \ell^{\beta,\Psi} (\delta x_0,\delta x_T)^2+\mathcal{O} \big( |(\delta x_0,\delta x_T)|^3,\\
f_i(x_t) &= f_{i,t} + Df_i \delta x_T +\half D^2 f_i \delta x_t ^2 + \mathcal{O} \big( |\delta x|^3 \big),\quad i=0,1, \\
g(x_t) &= \half g + Dg \delta x_t + \half g'' \delta x_t^2 + \mathcal{O} \big( |\delta x|^3 \big).
\end{split}
\ee
Observe that
\be
\label{Dell}
D\ell^{\beta,\Psi} (\delta x_0,\delta x_T) = [ p \, \delta x]^T_0 =
\int _0^T p(-f_x\delta x + \dot{\delta x}) {\rm d} t - \int_0^T g'(\xh) \delta x {\rm d} \mu.
\ee
Using the identities  \eqref{ellfg} and \eqref{Dell} in the following equation
\be
\LL(u,x,\lambda) = \ell^{\beta,\Psi}(x_0,x_T) +\int_0^T p \big( f_0+uf_1- \dot{x} \big) {\rm d} t + \int_0^T g(x) {\rm d} \mu,
\ee
we obtain the desired expression for $\LL(u,x,\lambda).$ The result follows.

\end{proof}

In view of the previous Lemma \ref{expansionL} and \cite[Lemma 8.4]{ABDL11} we can prove the following:

\begin{proposition}
\label{PropLexpansion}
Let $(v_k)\subset L^\infty$ be a sequence converging to 0 in the Pontryagin sense and $(x_{k,0})$ a sequence in $\cR^n$ converging to $\xh_0.$ Set $u_k := \uh+v_k,$ and let $x_k$ be the corresponding solution of equation \eqref{bsbstateeq} with initial value equal to $x_{k,0}.$ Then, for every $\lambda \in \Lambda,$ one has
\be
\label{Lexpansion}
\LL(u_k,x_k,\lambda) = \LL(\uh,\xh,\lambda) + \int_0^T H_u(t)v_{k,t} dt + Q(v_k,z_k,\lambda) + o(\gamma_k),
\ee
where $z_k:=z[v_k,x_{k,0}-\xh_0],$ $y_{k,t}:=\int_0^tv_{k,s} {\rm d}s$ and
$\gamma_k:=\gamma(y_k,y_{k,T},x_{k,0}-\xh_0).$
\end{proposition} 

\begin{remark}
Notice that Lemma 8.4 in \cite{ABDL11} was proved for $\gamma$ depending only on $(y,h)$ since the initial value $x_0$ was fixed throughout the article, while here $\gamma$ depends also on the initial variation of the state. The extension of Lemmas 8.3 and 8.4 in \cite{ABDL11} for the case with variable initial state are immediate, and the proofs are given in detail in \cite{Aro12ArXiv}.
\end{remark}

\begin{proposition}
\label{Propintg'}
Let $(p,\dd\mu) \in (BV)^{n*} \times \M$ verify the costate equation \eqref{costateeq}-\eqref{lmscac7}, and let $(v,z)\in L^2 \times (H^1)^n$ satisfy the linearized state equation \eqref{lineq}. Then,
\be
\label{eqintg'}
\int_0^T g'(\xh_t) z_t \dd \mu_{t}
+D\ell^{\beta,\Psi}(z_0,z_T) = \int_0^T H_u(t) v_t\dd t.
\ee
\end{proposition}

\begin{proof}
Note that
\be
\begin{split}
\int_0^T g'(\xh_t) z_t \dd \mu_{t}
&= - \int_0^T \sum_{j=1}^n z_{j,t} \dd p_{j,t}  - \int_0^T p_tf_x(\uh_t,\xh_t)z_t \dd t\\
&= - \int_0^T \sum_{j=1}^n z_{j,t} \dd p_{j,t} - \int_0^T p_t(\dot{z}_t -f_1(\xh_t)v_t)\dd t\\
&= -[pz]_{0-}^{T+} + \int_0^T H_u(t) v_t\dd t.
\end{split}
\ee
The wanted result follows from latter equation and the boundary conditions \eqref{lmscac7}.
\end{proof}

\begin{proof}
{\em (of Theorem \ref{SC}).}
Let us suppose on the contrary that  the conclusion does not hold. Hence, there should exist
a sequence $(v_k,x_{k,0})  \subset L^\infty\times \cR^n$ of non identically zero functions having $(v_k)$ converging to 0 in the Pontryagin sense and $x_{k,0} \to \xh_0,$ and such that, setting $u_k:=\uh+v_k,$ the corresponding solutions $x_k$ of \eqref{bsbstateeq} are feasible and
\begin{equation}
\label{chap1qgrowth}
\phi(x_{k,0},x_{k,T}) \leq \phi(\xh_0,\xh_T)  +o(\gamma_k),
\end{equation}
where $y_{k,t}:=\int_0^tv_{k,s} {\rm d}s$ and
$\gamma_k:=\gamma(y_k,y_{k,T},x_{k,0}-\xh_0).$
 Set $z_k:=z[v_k,z_{k,0}]$ with $z_{k,0}:= x_{k,0}-\xh_0.$  Take any $\lambda = (\beta, \Psi,p,\dd\mu) \in \Lambda,$ and multiply
inequality \eqref{chap1qgrowth} by $\beta$ (which is nonnegative), afterwards add the nonpositive term
$\Psi  \cdotp \Phi(x_{k,0},x_{k,T})+\int_0^Tg(x_k){\rm } \dd\mu$
to the left hand-side of the resulting inequality, and obtain the following:
\begin{equation}
\label{quadlag}
\LL(u_k,x_k,\lambda)\leq\LL(\uh,\xh,\lambda)+o(\gamma_k).
\end{equation}
Set
$(\yb_k,\hb_k,\xib_{k,0}):=(y_k,y_{k,T},z_{k,0})/\sqrt{\gamma_k}.$ Note that the
elements of this sequence have unit norm in $L^2\times
\cR\times \cR^n.$ By the Banach-Alaoglu Theorem, extracting if
necessary a subsequence, we may assume that there exists
$(\yb,\hb,\xib_0)\in L^2\times \cR \times \cR^n$  such that
\begin{equation}
\label{chap1limityk}
\yb_k\rightharpoonup \yb,\ \text{and}\
(\hb_k,\xib_{k,0})\rightarrow (\hb,\xib_0),
\end{equation}
where the first limit is given in the weak topology of $L^2.$ Define $\xib$ as  the solution of \eqref{eqxitab} associated to $\yb$ and the initial condition $\xib_0.$


The remainder of the proof is split in two parts:

{\em Fact 1:} The weak limit $(\yb,\hb,\xib)$ belongs to $\P^2_*.$

{\em Fact 2:} The inequality \eqref{quadlag} contradicts the hypothesis of uniform positivity \eqref{unifpos}.

 {\em Proof of Fact 1.} Recall the definition of the cone $\P^2_*$
 given in equation \eqref{tildeP2}. The proof of Fact 1 into four
 parts, in which we establish: 
{\em A)} condition $\eqref{intcond}(ii),
$ {\em B)}  \eqref{limitcond}(ii)-(iii), 
 {\em C)} \eqref{intcond}(i) and \eqref{limitcond}(iv), 
and {\em D)} that $(\yb,\hb,\xib)$ satisfies \eqref{transend}.

{\em A)} Let us show that $(\yb,\hb,\xib)$ verifies \eqref{intcond}(ii), i.e. $\yb$ is constant on each $B$ arc.
From \eqref{Lexpansion}, \eqref{quadlag} and the equivalence between $Q$ and $\Omega$ stated in Proposition \ref{QOm},
it follows that
\be
\label{boundHu}
-\Omega (\xi_k,y_k,h_k,\lambda)+o(\gamma_k) \geq  \int_0^TH_u(t)v_{k,t} \dtt \geq 0
,
\ee
where $\xi_k$ is solution of \eqref{eqxitab} corresponding to $y_k.$ The last inequality in \eqref{boundHu} holds in view of the minimum condition \eqref{lmscac4a} and since $\uh+v_k$ satisfies the control constraint \eqref{bsbcc}.
By the continuity of the mapping $\Omega$ over the space
$L^2\times \cR \times (H^1)^n $ and from \eqref{boundHu} we deduce that
\benl
0\leq \int_0^TH_u (t)v_{k,t}\dtt\leq O(\gamma_k).
\eenl
Hence, since the integrand in previous inequality is nonnegative for all $k\in \cN,$ we have that  
\be 
\label{cseq1}
\lim_{k\rightarrow \infty} \int_0^T H_u(t)\varphi_t \frac{v_{k,t}}{\sqrt{\gamma_k}}dt=0,
\ee
for any nonnegative $C^1$ function $\varphi:[0,T]\to \cR.$ 
Let us consider, in particular, such a function $\varphi$ having its support  included in a $B$ arc $(c,d).$  Integrating by parts in \eqref{cseq1} and in view of \eqref{chap1limityk}, we
obtain
\benl
\label{chap1cseq2}
0=\lim_{k\rightarrow \infty} \int_c^d
\ddt\left(H_u (t)\varphi_t\right)
\yb_{k,t}\dtt=\int_c^d
\ddt\left(H_u (t)\varphi_t \right)
\yb_t\dtt.
\eenl
Over $(c,d),$ $v_k$ has constant sign and therefore, $\yb$ is either nondecreasing or nonincreasing. Thus,  we can integrate by parts in the previous equation to get
\be
\label{chap1cseq3}
\int_c^dH_u(t)\varphi_t \mathrm{d} \yb_t=0.
\ee
Take now any $t_0\in(c,d).$ By the strict complementary condition for the control constraint assumed here (see Definition \ref{sc}), there exists a multiplier a multiplier $\hat\lambda=(\hat\beta,\hat\Psi,\hat p,\dd\hat{\mu})   \in \Lambda$ such that $H_u(\uh_{t_0},\xh_{t_0},\hat p_{t_0})>0.$ Hence, in view of the continuity of $H_u$ on $B,$\footnote{Actually $H_u$ is continuous on $B$ since $p$ does not jump on $B.$} there exists $\varepsilon>0$ such that $H_u(\uh_{t},\xh_{t},\hat p_{t})>0$  on $(t_0-2\varepsilon,t_0+2\varepsilon)\subset (c,d).$ Choose $\varphi$ such that $\supp\varphi\subset (t_0-2\varepsilon,t_0+2\varepsilon),$ and $H_u(\uh_{t},\xh_{t},\hat p_{t}) \varphi_t=1$ on $(t_0-\varepsilon,t_0+\varepsilon).$
Since $\mathrm{d}\yb\geq0,$ equation \eqref{chap1cseq3} yields
\be
\label{chap1cseq4}
\begin{split}
0
&=\int_c^d H_u(\uh_{t},\xh_{t},\hat p_{t}) \varphi_t \mathrm{d}\yb_t
\geq
\int_{t_0-\varepsilon}^{t_0+\varepsilon}
H_u(\uh_{t},\xh_{t},\hat p_{t}) \varphi_t \mathrm{d}\yb_t \\
&=
\int_{t_0-\varepsilon}^{t_0+\varepsilon} \mathrm{d}\yb_t
=\yb_{t_0+\varepsilon}-\yb_{t_0-\varepsilon}.
\end{split}
\ee
As both $\varepsilon$ and $t_0\in (c,d)$ are arbitrary we deduce that
\be
\label{chap1dybar0}
\mathrm{d}\yb_t=0,\quad \text{on } B.
\ee
Hence, $(\yb,\hb,\xib)$ satisfies \eqref{intcond}(ii).

{\em B)} Assume now that a $B_{0\pm}$ arc exists, and let us prove that $\yb = 0$ on $B_{0\pm}$ (see the definition of $B_{0\pm}$ in the paragraph preceding Proposition \ref{LimCond}). 
Let $B_{0\pm}$ be equal to $[0,t_1]$ for some $t_1>0.$ Assume without loss of generality that $\uh=u_{\min}$ on $[0,t_1].$
Notice that by the strict complementarity condition for the control constraint (condition (i) of the present theorem) there exists $\lambda' = (\beta',\Psi',p',\dd\mu') \in\Lambda$ and $\eps,\delta>0$ such that $H_u(\uh_{t},\xh_{t},p'_t)>\delta$ for all $t\in [0,\varepsilon],$ and thus, by considering in \eqref{cseq1} a nonnegative Lipschitz continuous function $\varphi:[0,T]\to \cR$ being equal to $1/\delta$ on $ [0,t],$ we obtain $\yb_{k,t}=\ds\int_0^{t} \frac{v_{k,s}}{\sqrt{\gamma_k}}\dd s\rightarrow 0,$ since $v_k\geq 0$ on $[0,t_1].$ 
Hence $\yb =0$ on $[0,\eps].$
This last assertion, together with \eqref{chap1dybar0}, imply that $\yb=0$ on $B_{0\pm}.$

Suppose that $T$ is in a boundary arc $B_{T\pm}.$ 
Let $B_{T\pm}=[t_N,T].$ Then, we can derive that for some $\eps>0,$ 
$\yb_{k,T}-\yb_{k,t} = \ds\int_{t}^T\frac{v_{k,s}}{\sqrt{\gamma_k}}\dd s\rightarrow 0$ for all $t\in [T-\eps,T],$  by an argument analogous to the one above. Thus, $\yb_t=\hb$ on $[T-\eps,T],$ and hence, by \eqref{chap1dybar0} we get that
\be
\label{chap1yihsc2}
\yb=\hb,\ \mathrm{on}\ (t_N,T],\ \mathrm{if}\ T\in B.
\ee
Therefore, $(\yb,\hb,\xib)$ verifies \eqref{limitcond}(ii)-(iii)

{\em C)} Let us prove that $(\yb,\hb,\xib)$ satisfies \eqref{intcond}(i) and \eqref{limitcond}(iv).
For all $t\in [0,T]$, a first order Taylor expansion gives
\be
0 \geq  g(x_{k,t}) =  g(\xh_t)+ g'(\xh_t)\delta x_{k,t} + O(|\delta x_{k,t}|^2).
\ee
From latter estimate and \cite[Lemma 8.12]{ABDL11}, we deduce that
\be
\label{gC}
g'(\xh_t) z_{k,t} \leq o(\sqrt\gamma_{k})+O(|\delta x_{k,t}|^2) ,\quad \text{for all } t\in C.
\ee
Let $\varphi \geq 0$ be some continuous function with support in $C$. 
From \eqref{gC}, we get that
\be
\begin{split}
\int_0^T \varphi_t  g'(\xh_t) (\xi_{k,t} + & f_1(\xh_t) y_{k,t}) \dd t
 = \int_0^T \varphi_t  g'(\xh_t) z_{k,t} \dd t \\
& \leq  \|\varphi\|_\infty \int_0^T \big( o(\sqrt\gamma_{k}) +O(|\delta x_{k,t}|^2)  \big) \dd t
\leq o(\sqrt\gamma_{k}),
\end{split}
\ee
where the last inequality follows from \cite[Lemma 8.4]{ABDL11}.
Therefore, dividing by  $\sqrt{\gamma_k}$ and
passing to the limit, we obtain
\be
\label{gnegaa}
\int_0^T \varphi_t 
g'(\xh_t)(\xib_t + f_1(\xh_t)\yb_t) \dd t \leq 0.
\ee
Since $\varphi$ is an arbitrary nonnegative continuous function with support
in $C$ (and $C$ is a finite union of intervals), we deduce that
\be
\label{gneg}
g'(\xh_t)(\xib_t + f_1(\xh_t)\yb_t)\leq 0,\quad \text{for a.a. } t \in C.
\ee
In particular, if $T\in C,$ we get from  \eqref{gC}:
\be
\label{gT1}
g'(\xh_T)(\xib_{T}+f_1(\xh_T)\hb) \leq 0\quad \text{if } T\in C.
\ee

Take  $(\beta,\Psi,p,\dd \mu) \in \Lambda.$ By Proposition \ref{Propintg'} and since $u_{\min} \leq \uh+v_k\leq u_{\max},$ we have
\be\label{Dell2}
D\ell^{\beta,\Psi} (\xh_0,\xh_T)(z_{k,0},z_{k,T}) + \int_0^T g'(\xh_t) z_{k,t} \dd \mu_t = \int_0^T H_u(t)v_{k,t} \dd t\geq 0.
\ee
On the other hand, a first order Taylor expansion of $\ell^{\beta,\Psi}$ and \cite[Lemma 8.12]{ABDL11} lead to
\be
D\ell^{\beta,\Psi} (\xh_0,\xh_T)(z_{k,0},z_{k,T}) = \ell^{\beta,\Psi} (x_{k,0},x_{k,T}) -\ell^{\beta,\Psi}(\xh_0,\xh_T) +o(\sqrt\gamma_k).
\ee
Hence, by \eqref{Dell2}, 
\be\label{deltaell}
\begin{split}
0 &\leq \ell^{\beta,\Psi}(x_{k,0},x_{k,T}) -\ell  ^{\beta,\Psi} (\xh_0,\xh_T) +o(\sqrt\gamma_k) + \int_0^T g'(\xh_t) z_{k,t} \dd \mu_t \\
& \leq\beta \phi(x_{k,0},x_{k,T}) -\beta\phi(\xh_0,\xh_T)+o(\sqrt\gamma_k) + \int_0^T g'(\xh_t) z_{k,t} \dd \mu_t,
\end{split}
\ee
where the last inequality holds since $\sum_{i=n_1+1}^{n_1+n_2} \Psi_i \Phi_i(x_{k,0},x_{k,T}) \leq 0.$
Observe now that, due to \eqref{chap1qgrowth}, $\beta\phi(x_{k,0},x_{k,T}) -\beta\phi(\xh_0,\xh_T) \leq o(\gamma_k).$ Hence, by latter estimate and from \eqref{deltaell}, we deduce that
\be
\label{intg'}
\frac{1}{\sqrt{\gamma_k}}\int_0^T g'(\xh_t) z_{k,t} \dd \mu_t= \int_0^T g'(\xh_t)(\xib_{k,t} + f_1(\xh_t)\yb_{k,t})  \dd \mu_t \geq o(1).
\ee
Since $\dd\mu_t$ has an essentially bounded density over $[0,T),$ we have that
\be
\label{intg'leq0} 
\begin{split}
0&\leq \liminf_{k\to \infty} \int_0^T  g'(\xh_t)(\xib_{k,t} + f_1(\xh_t)\yb_{k,t})  \dd \mu_t \\
&=  \lim_{k\to \infty} \Big( \int_{[0,T)} g'(\xh_t)(\xib_{k,t} + f_1(\xh_t)\yb_{k,t}) \dd \mu_t  + g'(\xh_T)(\xib_{k,T} + f_1(\xh_T)\yb_{k,T}) [\mu(T)]\big)\\
&= \int_{[0,T)} g'(\xh_t)(\xib_{t} + f_1(\xh_t)\yb_{t})\dd \mu_t + g'(\xh_T)(\xib_{T} +  f_1(\xh_T)\hb) [\mu(T)].
  \end{split}
  \ee
  In view of \eqref{gneg}, \eqref{gT1} and the complementary condition for the state constraint \eqref{weakcomp}(ii), we get from \eqref{intg'leq0} that
\be
g'(\xh_t) (\xib_t+f_1(\xh_t)\yb_t)=0,\quad \text{for a.a. } t \in C,
\ee
and
\be
g'(\xh_T)(\xib_{T} +  f_1(\xh_T)\hb)=0,\quad \text{whenever } T\in C\text{ and  } [\mu(T)] >0.
\ee
Then, $(\yb,\hb,\xib)$ satisfies \eqref{intcond}(i). Furthermore, if $T\in $ and $[\mu(T)] >0$ for some $\lambda\in \Lambda,$ it holds
\be
\hb = -\frac{g'(\xh_T)\xib_{T}}{g'(\xh_T)f_1(\xh_T)} = \lim_{t \uparrow T} -\frac{g'(\xh_t)\xib_{t}}{g'(\xh_t)f_1(\xh_t)} =  \lim_{t \uparrow T} \yb(t),
\ee
so that \eqref{limitcond}(iv) holds.

{\em D)} We shall now prove \eqref{transend}.
Let $i=1,\dots,n_1+n_2,$ then
\be
\label{Phi'}
\begin{split}
\Phi_i' (\xh_0,\xh_T)(\xib_0,\xib_T+f_1(\xh_T)\hb)
&= \lim_{k\to\infty} \Phi_i' (\xh_0,\xh_T) \frac{(\xi_{k,0},\xi_{k,T} + f_1(\xh_T) y_{k,T})}{\sqrt{\gamma_k}} \\
&= \lim_{k\to\infty} \Phi_i' (\xh_0,\xh_T) \frac{(z_{k,0},z_{k,T}) }{\sqrt{\gamma_k}}.
\end{split}
\ee
A first order Taylor expansion of $\Phi_i$ at $(\xh_0,\xh_T)$ and \cite[Lemma 8.12]{ABDL11} yield
\be
\label{Phi'z}
\Phi_i' (\xh_0,\xh_T) \frac{(z_{k,0},z_{k,T}) }{\sqrt{\gamma_k}}
= \frac{\Phi_i(x_{k,0},x_{k,T}) - \Phi_i (\xh_0,\xh_T)}{\sqrt{\gamma_k}} + o(1).
\ee
Thus, from \eqref{Phi'}-\eqref{Phi'z} we
get
\begin{gather*}
\Phi_i' (\xh_0,\xh_T)(\xib_0,\xib_T+f_1(\xh_T)\hb) = 0,\quad \text{for } i=1,\dots,n_1,\\
\Phi_i' (\xh_0,\xh_T)(\xib_0,\xib_T+f_1(\xh_T)\hb) \leq 0,\,\, \text{for } i=n_1+1,\dots,n_1+n_2,\ \text{with } \Phi_i (\xh_0,\xh_T) =0.
\end{gather*}
For the endpoint cost, we can obtain analogous expressions to \eqref{Phi'}-\eqref{Phi'z}, and then, from \eqref{chap1qgrowth} we get
\be
\phi' (\xh_0,\xh_T)(\xib_0,\xib_T+f_1(\xh_T)\hb) \leq 0.
\ee
Hence, \eqref{transend} is verified.

We conclude that $(\yb,\hb,\xib) \in \P^2_*,$ this is, {\em Fact 1} follows.

 {\em Proof of Fact 2.}
From equation \eqref{Lexpansion} in Proposition \ref{PropLexpansion} we obtain that
\be
\Omega(y_k,h_k,\xi_k,\lambda) =
\LL (u_k,x_k,\lambda)- \LL(\uh,\xh,\lambda) - \int_0^T H_u(t) v_{k,t} \dtt+o(\gamma_k) \leq o(\gamma_k),
\ee
where the last inequality follows from \eqref{quadlag} and since $H_u(t)v_{k,t} \geq 0,$ a.e. on $[0,T].$ Hence,
\be
\label{liminfOm}
\liminf_{k\to \infty} \Omega (y_k,h_k,\xi_k,\lambda) \leq
\limsup_{k\to \infty} \Omega (y_k,h_k,\xi_k,\lambda) \leq 0.
\ee

Let us recall that, for each $\lambda \in \Lambda,$ the mapping $\Omega(\cdot,\lambda)$ is a Legendre form in the Hilbert space $\{(y,h,\xi)\in L^2 \times \cR \times (H^1)^n:\text{\eqref{eqxitab} holds}\}$ (in view of hypothesis (iii) of the current theorem). In particular, for $\bar\lambda\in \Lambda$ reaching the maximum in \eqref{unifpos} for the critical direction $(\bar y,\hb,\xib),$ one has
\be
\label{gamma0}
\rho \gamma(\yb,\hb,\xib_0) \leq \Omega (\yb,\hb,\xib,\bar\lambda) = 
\liminf \Omega (\yb_k,\hb_k,\xib_k,\bar\lambda)\leq 0,
\ee
where the equality holds since $\Omega$ is a Legendre form and the
last inequality follows from \eqref{liminfOm}.
In view of \eqref{gamma0}, we get that $(\yb,\hb,\xib_0)=0$ 
and $\lim  \Omega(\yb_k,\hb_k,\xib_k,\bar\lambda)=0.$ Consequently, 
$(\yb_k,\hb_k,\xib_{k,0})$ converges strongly to $(\yb,\hb,\xib_0)=0,$ which is a contradiction since $(\yb_k,\hb_k,\xib_{k,0})$ has unit norm in $L^2 \times \cR \times \cR^n.$ We conclude that $(\uh,\xh)$ is a Pontryagin minimum satisfying $\gamma-$growth in the Pontryagin sense.
\end{proof}

\section*{Acknowledgments}

We wish to thank the anonymous referees for their bibliographical advices. 

This work was partially supported by the European Union under the 7th Framework Programme FP7-PEOPLE-2010-ITNÐ Grant agreement number 264735-SADCO. The last stage of this research took place while the first author was holding a postdoctoral position at IMPA, Rio de Janeiro, with CAPES-Brazil funding.

\appendix

\section{On second-order necessary conditions}

\subsection{General constraints}\label{tool_ocfo}
We will study an abstract optimization problem of the form 
\be
\label{P0}
\min f(x); \quad G_E(x)=0, \,\, G_I(x) \in K_I,
\ee
where $X$, $Y_E$, $Y_I$ are Banach spaces, $f:X\rar\cR$,
$G_E:X\rar Y_E$, and $G_I:X\rar Y_I$
are functions of class $C^2$, 
and $K_I$ is a closed convex subset of $Y_I$ with nonempty interior.
The subindex $E$ is used to refer to `equalities' and $I$ to `inequalities'.

Setting
\be
\label{tool_ocfo2}
Y:= Y_E\times Y_I, \quad G(x):=(G_E(x),G_I(x)),
\quad K:= \{0\}_{K_E} \times K_I, 
\ee
we can rewrite problem \eqref{P0} in the more compact form
\be
\label{Pabs}\tag{$P_A$}
\min f(x); \quad G(x) \in K.
\ee
We use $F(P_A)$ to denote the set of feasible solutions of
\eqref{Pabs}.

\begin{remark}
\label{tool_ocfo.r1}
We refer to \cite{PAOP} for a 
systematic study of problem $\eqref{Pabs}.$
Here we will take advantage of the product structure (that
one can find in essentially all practical applications) 
 to introduce a {\em non qualified version} of 
second order necessary conditions specialized to the case of quasi radial directions, that extends in some sense \cite[Theorem 3.50]{PAOP}.
See Kawasaki \cite{Kaw88} for non radial directions. 
\end{remark}

The {\em tangent cone} (in the sense of convex analysis) to $K_I$ at $y\in K_I$ is defined as
\be
T_{K_I}(y) := \{ z\in Y_I: \dist(y+ t z,K_I)=o(t), \;\; \text{with } t\geq 0\},
\ee
and the {\em normal cone} to $K_I$ at $y\in K_I$ is 
\be
N_{K_I}(y) := \{ z^* \in Y_I^*: \la z^*, y'-y \ra \leq 0,\,\, \text{for all } y'\in K_I \}.
\ee

In what follows, we shall study a nominal feasible solution $\xh \in F(P_A)$
that may satisfy or not the {\em qualification condition}
\be
\label{qualif}
\left\{ \ba{lll}
{\rm (i)} & \text{$DG_E(\xh)$ is onto,}
\\
{\rm (ii)} & \text{there exists $z \in \Ker DG_E(\xh)$ such that }
            \text{ $G_I(\xh) + DG_I(\xh)z \in \intt(K_I)$. }
\ea\right.
\ee 
The latter condition coincides with the qualification condition in \eqref{tool_ocfo.qualifoc} which was introduced for the optimal control problem \eqref{P}.

\begin{remark}
\label{tool_ocfo.r3}
Condition \eqref{qualif}
is equivalent to 
\if {the existence of $\eps>0$ such that 
\be
\label{tool_ocfo.qualifROB}
\eps B_Y \subset G(\xh) + DG(\xh)X - K.
\ee
This is, for the general format \eqref{tool_ocfo3},
} \fi
the Robinson qualification condition in \cite{Rob76a}.
See the discussion in \cite[Section 2.3.4]{PAOP}. 

\end{remark}

The {\em Lagrangian function} of problem \eqref{Pabs} is defined as
\be
L(x,\lambda) := \beta f(x) + \la \lambda_E, G_E(x) \ra
+ \la \lambda_I, G_I(x) \ra,
\ee
where we set $\lambda :=(\beta,\lambda_E,\lambda_I)
\in \RR_+ \times Y_E^* \times Y_I^*.$
Define the {\em set of Lagrange multipliers}
\index{Lagrange multipliers}
associated with $ x \in F(P_A)$ as 
\be
\Lambda(x) := \{  \lambda \in \RR_+ \times Y_E^* \times N_{K_I}(G_I(x)): \lambda\neq0, \,D_x L(x,\lambda) = 0 \}. 
\ee

Let $y_I \in \intt(K_I),$ $y_I \neq G_I(\xh).$ 
We consider the following auxiliary problem,
where  $(x,\gamma) \in X \times \RR$:
\be
\label{APabs}\tag{$AP_A$}
\ba{lll} \disp
\min_{x,\gamma} \gamma; &
f(x) - f(\xh) \leq \gamma, \quad 
G_E(x)=0, \quad \gamma\geq -1/2,
\\ &
y_I + (1+\gamma)^{-1} \big( G_I(x) - y_I \big) \in K_I. 
\ea
\ee
Note that we recognize the idea of a {\em gauge function} (see e.g. \cite{Roc70}) in the 
last constraint. 

\begin{lemma}
\label{tool_ocfo.l-ap}
Assume that $\xh$ is a local solution of \eqref{Pabs}.
Then $(\xh,0)$ is a local solution of \eqref{APabs}.
\end{lemma}

\begin{proof}
We easily check that 
$(\xh,0)\in F(AP_A)$.
Now take $(x,\gamma)\in F(AP_A).$
Let us prove that if $-1/2 \leq \gamma<0,$ then $x$ cannot be closed to $\xh$ (in the norm of the Banach space $X$).
Assuming that  $-1/2 \leq \gamma<0,$ we get $G_E(x)=0,$
$G_I(x) \in K_I + (-\gamma) y_I \subseteq K_I$,
and 
$f(x) < f(\xh)$. Since 
$\xh$ is a local solution of \eqref{Pabs},
the $x$ cannot be too closed to $\xh.$ The conclusion follows. 
\end{proof}

The Lagrangian function of \eqref{APabs}, in 
qualified form,  is 
\be
\gamma + 
\beta( f(x) - f(\xh) - \gamma) 
+ \la \lambda_E, G_E(x) \ra
+ \la \lambda_I, 
y_I + (1+\gamma)^{-1} (G_I(x) - y_I) \ra.
\ee
or equivalently
\be
L(x,\lambda) + (\beta_0 - \beta) \gamma
+ \big((1+\gamma)^{-1} -1 \big) \la \lambda_I,  G_I(x) - y_I\ra.
\ee 
Setting 
$\hat\lambda=(\beta_0,\beta,\lambda_E,\lambda_I)$,
we see that the set of Lagrange multipliers 
of the auxiliary problem  \eqref{APabs}
at $(\xh,0)$ is 
\be
\label{hatLambda}
\hat\Lambda := 
\left\{ \ba{lll}
 \hat\lambda 
 \in \RR_+ \times \RR_+ \times Y_E^* \times N_{K_I}(G_I(\xh)): \lambda \neq 0, 
\\
D_x L(\xh,\lambda) =0; \; 
\beta + \la \lambda_I,  G_I(\xh) - y_I \ra = 1
\ea \right\}.
\ee

\begin{proposition}
\label{propderiv.p}
Suppose that \eqref{qualif}(i) holds. Then, the mapping
\be
\label{maplambda}
(\beta,\lambda_E,\lambda_I) 
\mapsto
\frac{(\beta + \la \lambda_I,  G_I(x) - y_I \ra,
\beta,\lambda_E,\lambda_I) }{\beta + \la \lambda_I,  G_I(x) - y_I \ra}
\ee
is a bijection between 
$\Lambda(\xh)$
and 
$\hat\Lambda_1$ (recall the definition  in  \eqref{LambdaBeta}).
\end{proposition}

\begin{proof}
Since \eqref{qualif}(i) holds, then we necessarily have that $(\beta,\lambda_I) \neq 0$ for all $\lambda=(\beta,\lambda_E,\lambda_I) \in \Lambda(\xh).$ Therefore, if $\lambda_I = 0$ then $\beta>0$ and $\beta + \la \lambda_I,  G_I(x) - y_I \ra > 0.$ If by the contrary, $\lambda_I \neq 0,$ then $\la \lambda_I,  G_I(x) - y_I \ra > 0$ and again, $\beta + \la \lambda_I,  G_I(x) - y_I \ra > 0.$
Hence, the mapping in \eqref{maplambda} is well-defined and is a bijection from $\Lambda(\xh)$ to $\hat\Lambda_1,$ as we wanted to show.
\end{proof}

\begin{theorem}
\label{tool_ocfo.t2}
Let $\xh$ be a local solution of \eqref{Pabs},
such that $DG_E(\xh)$ is surjective.
Then $\hat\Lambda_1$ is non empty and bounded. 
\end{theorem}

\begin{proof}
By lemma \ref{tool_ocfo.l-ap},
$(\xh,0)$ is a local solution of 
\eqref{APabs}.  In addition the qualification condition for the latter problem
 at the point 
$(\xh,0)$ states as follows: there exists $(z,\delta) \in \Ker DG_E(\xh) \times \cR$ such that
\begin{gather*}
Df(\xh)z<\delta,\quad \delta>0,\\
G_I(\xh) +  D G_I(\xh)z -\delta(G_I(\xh) - y_I) \in \intt(K_I).
\end{gather*}
These conditions trivially hold for 
$(z,\delta)=(0,1).$
Hence, in view of classical results by e.g. Robinson \cite{Rob76a},
the conclusion follows. 
\end{proof}

\subsection{Second order necessary optimality conditions}\label{tool_ocso}

Let us introduce the notation $[a,b]$ to refer to the segment $\{\rho a+(1-\rho)b;\ \text{for }\rho \in [0,1]\},$ defined for any pair of points $a,b$ in an arbitrary vector space $Z.$

\begin{definition}
\label{tool_ocso.d1}
Let $y\in K$. We say that $z\in Y$ is a
{\em radial direction} to $K$ at $y$ if 
$[y,y+\eps z] \subset K$
for some $\eps >0$,
and a 
{\em quasi-radial direction} if 
$\dist(y+\sigma z,K) = o(\sigma^2)$ for $\sigma>0$. 
\end{definition}

Note that any radial direction is also quasi-radial, and both radial and quasi radial directions are tangent.
With $\xh\in F(P_A)$, we associate the {\em critical cone}\index{critical cone}
\be
\label{tool_ocso1}
C(\xh) := \{z\in X: Df(\xh)z \leq 0, \; DG_E(\xh)z = 0, \; DG_I(\xh)z \in T_K(G_I(\xh)) \}. 
\ee

\begin{definition}
\label{tool_ocso.d2}
We say that $z\in C(\xh)$ is a {\em radial (quasi radial)
critical direction} for problem \eqref{Pabs} if 
$D G_I(\xh) z$ is a radial (quasi radial) direction 
to $K_I$ at $G_I(\xh)$.
We write $C_{QR}(\xh)$ for the {\em set of quasi radial
critical directions.}
The critical cone $C(\xh)$ is 
{\em quasi radial} if $C_{QR}(\xh)$ is 
a dense subset of $C(\xh)$.
\end{definition}

It is immediate to check that  $C_{QR}(\xh)$ is a convex cone. 

We next state {\em primal second order necessary conditions} for the problem \eqref{Pabs}. 
Consider the following optimization problem,
where $z\in X$, $w\in X$ and $\theta\in\cR$:
\be
\label{Qz}
\tag{$Q_z$}
\left\{ 
\begin{split}
&\min_{w,\theta} \;\; \theta,
\\
& Df(\xh) w + D^2f(\xh)(z,z) \le \theta,
\\ & D G_E(\xh) w + D^2 G_E(\xh)(z,z) = 0,
\\
& D G_I(\xh) w + D^2 G_I(\xh)(z,z)  
- \theta (G(\xh)-y_I) \in T_K(G_I(\xh)).
\end{split}
\right.
\ee

\begin{theorem}
\label{tool_ocso.t1}
Let $(\xh,0)$ be a local solution of \eqref{APabs},
such that $DG_E(\xh)$ is surjective,
and let  
$h\in C_{QR}(\xh)$.
Then problem 
\eqref{Qz} is feasible, and has a nonnegative value. 
\end{theorem}

\begin{proof}
We shall first show that \eqref{Qz} is feasible. 
Since
$DG_E(\xh)$ is surjective, there exists 
$w \in X$ such that $D G_E(\xh) w + D^2 G_E(\xh)(z,z) = 0$.
Since $T_K(G_I(\xh))$ is a cone, the last equation 
divided by $\theta>0$ is equivalent to
\be 
\label{thetamungxh}
\theta ^{-1} \big(D G_I(\xh) w + D^2 G_I(\xh)(z,z)\big) 
+ y_I - G(\xh)\in T_K(G_I(\xh)).
\ee
Since $y_I\in\intt(K_I)$,
we have that 
$y_I - G(\xh)\in\intt T_K(G_I(\xh))$, and therefore 
the last constraint of \eqref{Qz} holds
when $\theta$ is large enough. 
So it does the first constraint,
and hence, \eqref{Qz} is feasible. 

We next have to show that we cannot have 
 $(w,\theta_0)\in F(Q_z)$ with $\theta_0<0.$ Let us suppose, on the contrary, that there is such a feasible solution $(w,\theta_0).$
Set $\theta:=\half \theta_0$. 
Then 
$Df(\xh) w + D^2f(\xh)(z,z) < \theta.$
Using \eqref{thetamungxh} and
$y_I\in\intt(K_i)$,
 we can easily show that,
for some $\eps>0$: 
\be
\label{gI2yI}
D G_I(\xh) w + D^2 G_I(\xh)(z,z)  
- \theta (G_I(\xh)-y_I) + \eps B \in T_K(G_I(\xh)).
\ee

Consider, for $\sigma>0$, the path
\be
x_\sigma := \xh + \sigma z + \half \sigma^2 w.
\ee
By a second order Taylor expansion we obtain that 
$G_E(x_\sigma) = o(\sigma^2)$. 
Since $DG_E(\xh)$ is onto,
by Lyusternik's theorem \cite{lyu34}, 
there exists a 
path $x'_\sigma = x_\sigma + o(\sigma^2)$,
such that $G_E(x'_\sigma)=0$. 
Assuming, without loss of generality, that $G_I(\xh)=0$, we get 
\be
\label{tool_ocso.t1g}
G_I(x'_\sigma)  = \sigma D G_I(\xh)z+ \half \sigma^2 
\left[ D G_I(\xh) w + D^2 G_I(\xh)(z,z) \right] + o(\sigma^2).
\ee
Setting
\be
\left\{ 
\begin{split}
k_1(\sigma) &:= (1-\sigma)^{-1}\sigma D G_I(\xh)z,
\\
k_2(\sigma) &:= \sigma \big( D G_I(\xh) w + D^2 G_I(\xh)(z,z) \big),
\end{split}
\right.
\ee
we can rewrite \eqref{tool_ocso.t1g} as 
\be
G_I(x'_\sigma)  = (1-\sigma) k_1(\sigma) + \half \sigma k_2(\sigma) 
 + o(\sigma^2).
\ee
Since $z$ is a quasi radial critical direction, there exists 
$k'_1(\sigma)\in K_I$ such that 
\be
\sigma D G_I(\xh)z = k'_1(\sigma)
+ o(\sigma^2),
\ee
and so,
\be
G_I(x'_\sigma)  \in (1-\sigma) K_I 
 + \half \sigma k_2(\sigma) 
 + o(\sigma^2).
\ee
Using  \eqref{gI2yI} and $G_I(\xh)=0$
we obtain
\be
k_2(\sigma) + \sigma \theta y_I + 
\sigma \eps B \in K_I.
\ee
Therefore, for $\sigma>0$ small enough
\be
G_I(x'_\sigma) 
\in (1-\half\sigma) K_I 
 + \half \sigma K_I - \half\sigma^2(\theta y_I +o(1))
\subset K_I,
\ee
where we have used the fact that since 
$0=G_I(\xb)\in K_I$, we have that 
(remember that $\theta<0$):
$\half\sigma^2 (-\theta) (y_I + \eps B) \subset K_I$.

We check easily that 
$f(x'_\sigma) <0$,
and so, we have constructed a feasible path for \eqref{APabs},  contradicting the local optimality of
$(\xh,0)$. 

We conclude that such a solution $(w,\theta_0)$ of \eqref{Qz} with $\theta_0<0$ cannot exist and, therefore, \eqref{Qz} has nonnegative value.
\end{proof}

We now present dual second order necessary conditions. 
\if{
Consider the problem
\be
\label{tool_ocso.t2_5}
\max_{\lambda\in \Lambda(\xh)} D^2_{xx} L(\xh,\lambda)(z,z). 
\ee
} \fi

\begin{theorem}
\label{appt2}
Let $\xh$ be a local minimum of \eqref{Pabs},
that satisfies the qualification condition
\eqref{qualif}.
Then, for every  $z\in C_{QR}(\xh),$ 
\be
\label{tool_ocso.t2_1}
\max_{\lambda \in \Lambda(\xh)} D^2_{xx} L(\xh,\lambda)(z,z) \geq 0.
\ee
\end{theorem}

\begin{proof}
Since  problem \eqref{Qz} is qualified with a
finite nonnegative value, 
by the convex duality theory  \cite{EkeTem},
its dual has a nonnegative value
and a nonempty set of solutions. 
The Lagrangian of problem \eqref{Qz} 
in qualified form ($\beta_0=1$)
can be written as 
\be
D_xL(\xh,\lambda)w + D^2_{xx} L(\beta,\xh,\lambda)(z,z)+
\big(1-\beta + \la\lambda_I,G(\xh)-y_I\ra \big) \theta,
\ee
where 
$\lambda=(\beta,\lambda_E,\lambda_I)$
as before,  and so, the dual problem of \eqref{Qz}
can be written as 
$$
\Max_{\lambda\in\Lambda(\xh)} 
D^2_{xx} L(\beta,\xh,\lambda)(z,z);
\quad
\beta + \la\lambda_I,G(\xh)-y_I\ra) \theta=1.
$$
The conclusion follows.
\end{proof}

\begin{remark}
Whereas the above theorem follows from Cominetti \cite{Com90} or Kawasaki \cite{Kaw88}, 
our proof avoids the concepts of 
second order tangent set and its associated calculus, 
used in these references. This considerably simplifies
the proof.
\end{remark}

\bibliographystyle{plain}
\bibliography{diopt,hjb,paop,stateconstraints,shooting,singulararc}

\def\cprime{$'$} \def\cprime{$'$} \def\cprime{$'$} \def\cprime{$'$}
  \def\cprime{$'$}
\begin{thebibliography}{10}

\bibitem{Ada75}
R.A. Adams.
\newblock {\em Sobolev spaces}.
\newblock Academic Press, New York, 1975.

\bibitem{AftBon14}
A.~Aftalion and J.~Bonnans.
\newblock Optimization of running strategies based on anaerobic energy and
  variations of velocity.
\newblock {\em SIAM J. Applied Math.}, 74(5):1615--1636, 2014.

\bibitem{AgSachkov}
A.~Agrachev and Y.~Sachkov.
\newblock {\em Control theory from the geometric viewpoint}.
\newblock Encyclopaedia of Mathematical Sciences, 87, Control Theory and
  Optimization, II. Springer-Verlag, Berlin, 2004.

\bibitem{AgrSteZez02}
A.A. Agrachev, G.~Stefani, and P.L. Zezza.
\newblock Strong optimality for a bang-bang trajectory.
\newblock {\em SIAM J. Control Optim.}, 41:991--1014, 2002.

\bibitem{Aro12ArXiv}
M.S. Aronna.
\newblock Singular solutions in optimal control: second order conditions and a
  shooting algorithm.
\newblock Technical report, ArXiv, 2013.
\newblock [Published online as arXiv:1210.7425, submitted].

\bibitem{ABDL11}
M.S. Aronna, J.F. Bonnans, A.V. Dmitruk, and P.A. Lotito.
\newblock Quadratic order conditions for bang-singular extremals.
\newblock {\em Numer. {A}lgebra, {C}ontrol {O}ptim., special issue dedicated to
  {P}rofessor {H}elmut {M}aurer on the occasion of his 65th birthday},
  2(3):511--546, 2012.

\bibitem{MR780283}
A.V. Arutyunov.
\newblock On necessary conditions for optimality in a problem with phase
  constraints.
\newblock {\em Dokl. Akad. Nauk SSSR}, 280(5):1033--1037, 1985.

\bibitem{MR1845332}
A.V. Arutyunov.
\newblock {\em Optimality conditions}, volume 526 of {\em Mathematics and its
  Applications}.
\newblock Kluwer Academic Publishers, Dordrecht, 2000.
\newblock Abnormal and degenerate problems, Translated from the Russian by S.
  A. Vakhrameev.

\bibitem{BonHer09}
J.F. Bonnans and A.~Hermant.
\newblock Revisiting the analysis of optimal control problems with several
  state constraints.
\newblock {\em Control Cybernet.}, 38(4A):1021--1052, 2009.

\bibitem{PAOP}
J.F. Bonnans and A.~Shapiro.
\newblock {\em Perturbation analysis of optimization problems}.
\newblock Springer-Verlag, New York, 2000.

\bibitem{Bonnard2003}
B.~Bonnard, L.~Faubourg, G.~Launay, and E.~Tr\'elat.
\newblock Optimal control with state constraints and the space shuttle re-entry
  problem.
\newblock {\em Journal of Dynamical and Control Systems}, 9(2):155--199, 2003.

\bibitem{Com90}
R.~Cominetti.
\newblock Metric regularity, tangent sets and second order optimality
  conditions.
\newblock {\em J. Appl. Math. Optim.}, 21:265--287, 1990.

\bibitem{dePinho2005}
M.R. de~Pinho, M.M. Ferreira, U.~Ledzewicz, and H.~Schaettler.
\newblock A model for cancer chemotherapy with state-space constraints.
\newblock {\em Nonlinear Analysis: Theory, Methods \& Applications},
  63(5):e2591--e2602, 2005.

\bibitem{EkeTem}
I.~Ekeland and R.~Temam.
\newblock {\em Convex analysis and variational problems}, volume~1 of {\em
  Studies in Mathematics and its Applications}.
\newblock North-Holland, Amsterdam, 1976.
\newblock French edition: Analyse convexe et probl\`emes variationnels, Dunod,
  Paris, 1974.

\bibitem{Fel04}
U.~Felgenhauer.
\newblock Optimality and sensitivity for semilinear bang-bang type optimal
  control problems.
\newblock {\em Int. J. Appl. Math. Comput. Sci.}, 14(4):447--454, 2004.

\bibitem{FraTon13}
H.~Frankowska and D.~Tonon.
\newblock Pointwise second-order necessary optimality conditions for the
  {M}ayer problem with control constraints.
\newblock {\em SIAM J. Control Optim.}, 51(5):3814--3843, 2013.

\bibitem{GabKir72}
R.~Gabasov and F.M. Kirillova.
\newblock High-order necessary conditions for optimality.
\newblock {\em J. SIAM Control}, 10:127--168, 1972.

\bibitem{Goh66}
B.S. Goh.
\newblock Necessary conditions for singular extremals involving multiple
  control variables.
\newblock {\em J. SIAM Control}, 4:716--731, 1966.

\bibitem{GLV74}
B.S. Goh, G.~Leitmann, and T.L. Vincent.
\newblock Optimal control of a prey-predator system.
\newblock {\em Math. Biosci.}, 19:263--286, 1974.

\bibitem{GraPet08}
K.~Graichen and N.~Petit.
\newblock Solving the {G}oddard problem with thrust and dynamic pressure
  constraints using saturation functions.
\newblock In {\em 17th World Congress of The International Federation of
  Automatic Control}, volume Proc. of the 2008 IFAC World Congress, pages
  14301--14306, Seoul, 2008. IFAC.

\bibitem{Hes51}
M.R. Hestenes.
\newblock Applications of the theory of quadratic forms in {H}ilbert space to
  the calculus of variations.
\newblock {\em Pacific J. Math.}, 1(4):525--581, 1951.

\bibitem{Hof52}
A.~Hoffman.
\newblock On approximate solutions of systems of linear inequalities.
\newblock {\em Journal of Research of the National Bureau of Standards, Section
  B, Mathematical Sciences}, 49:263--265, 1952.

\bibitem{JacSpe71}
D.H. Jacobson and Speyer J.L.
\newblock Necessary and sufficient conditions for optimality for singular
  control problems: {A} limit approach.
\newblock {\em J. Math. Anal. Appl.}, 34:239--266, 1971.

\bibitem{MR2376656}
D.~Yu. Karamzin.
\newblock Necessary conditions for an extremum in a control problem with phase
  constraints.
\newblock {\em Zh. Vychisl. Mat. Mat. Fiz.}, 47(7):1123--1150, 2007.

\bibitem{Kaw88}
H.~Kawasaki.
\newblock An envelope-like effect of infinitely many inequality constraints on
  second order necessary conditions for minimization problems.
\newblock {\em Math. Program.}, 41:73--96, 1988.

\bibitem{Kel64}
H.J. Kelley.
\newblock A second variation test for singular extremals.
\newblock {\em AIAA Journal}, 2:1380--1382, 1964.

\bibitem{lyu34}
L.~Lyusternik.
\newblock Conditional extrema of functions.
\newblock {\em Math. USSR-Sb}, 41:390--40, 1934.

\bibitem{MalMau96}
K.~Malanowski and H.~Maurer.
\newblock Sensitivity analysis for parametric control problems with
  control-state constraints.
\newblock {\em Comput. Optim. Appl.}, 5(3):253--283, 1996.

\bibitem{Maurer77}
H.~Maurer.
\newblock On optimal control problems with bounded state variables and control
  appearing linearly.
\newblock {\em SIAM J. Control Optim.}, 15(3):345--362, 1977.

\bibitem{MauKimVos05}
H.~Maurer, J.-H.R. Kim, and G.~Vossen.
\newblock On a state-constrained control problem in optimal production and
  maintenance.
\newblock In C.~Deissenberg and R.F. Hartl, editors, {\em Optimal Control and
  Dynamic Games}, volume~7 of {\em Advances in Computational Management
  Science}, pages 289--308. Springer US, 2005.

\bibitem{MauOsm03b}
H.~Maurer and N.P. Osmolovskii.
\newblock Second order optimality conditions for bang-bang control problems.
\newblock {\em Control Cybernet.}, 32:555--584, 2003.

\bibitem{MacPow71}
J.P. McDanell and W.F. Powers.
\newblock Necessary conditions joining optimal singular and nonsingular
  subarcs.
\newblock {\em SIAM J. Control}, 9:161--173, 1971.

\bibitem{MilOsm98}
A.A. Milyutin and N.~N. Osmolovskii.
\newblock {\em Calculus of Variations and Optimal Control}.
\newblock American Mathematical Society, Providence, 1998.

\bibitem{MauOsm05}
N.~P. Osmolovskii and H.~Maurer.
\newblock Equivalence of second order optimality conditions for bang-bang
  control problems. {I}. {M}ain results.
\newblock {\em Control Cybernet.}, 34(3):927--950, 2005.

\bibitem{MR3012263}
N.P. Osmolovskii and H.~Maurer.
\newblock {\em Applications to regular and bang-bang control}, volume~24 of
  {\em Advances in Design and Control}.
\newblock Society for Industrial and Applied Mathematics (SIAM), Philadelphia,
  PA, 2012.
\newblock Second-order necessary and sufficient optimality conditions in
  calculus of variations and optimal control.

\bibitem{PogSpa11}
L.~Poggiolini and M.~Spadini.
\newblock Strong local optimality for a bang-bang trajectory in a {M}ayer
  problem.
\newblock {\em SIAM J. Control Optim.}, 49:140--161, 2011.

\bibitem{PogSte11}
L.~Poggiolini and G.~Stefani.
\newblock Bang-singular-bang extremals: sufficient optimality conditions.
\newblock {\em J. Dyn. Control Syst.}, 17(4):469--514, 2011.

\bibitem{MR1814264}
F.~Rampazzo and R.~Vinter.
\newblock Degenerate optimal control problems with state constraints.
\newblock {\em SIAM J. Control Optim.}, 39(4):989--1007 (electronic), 2000.

\bibitem{Rob76a}
S.M. Robinson.
\newblock First order conditions for general nonlinear optimization.
\newblock {\em SIAM J. Appl. Math.}, 30:597--607, 1976.

\bibitem{Roc70}
R.T. Rockafellar.
\newblock {\em Convex Analysis}.
\newblock Princeton University Press, Princeton, New Jersey, 1970.

\bibitem{MR0397507}
I.~B. Russak.
\newblock Second-order necessary conditions for general problems with state
  inequality constraints.
\newblock {\em J. Optimization Theory Appl.}, 17(112):43--92, 1975.

\bibitem{MR0390876}
I.B. Russak.
\newblock Second order necessary conditions for problems with state inequality
  constraints.
\newblock {\em SIAM J. Control}, 13:372--388, 1975.

\bibitem{Schaettler2005}
H.~Schattler.
\newblock A local field of extremals near boundary arc - interior arc
  junctions.
\newblock In {\em Decision and Control, 2005 and 2005 European Control
  Conference. CDC-ECC '05. 44th IEEE Conference on}, pages 945--950, Dec 2005.

\bibitem{Schaettler2006}
H.~Sch{\"a}ttler.
\newblock Local fields of extremals for optimal control problems with state
  constraints of relative degree 1.
\newblock {\em J. Dyn. Control Syst.}, 12(4):563--599, 2006.

\bibitem{SeyCli93}
H.~Seywald and E.M. Cliff.
\newblock {G}oddard problem in presence of a dynamic pressure limit.
\newblock {\em Journal of Guidance, Control, and Dynamics}, 16(4):776--781,
  1993.

\end{thebibliography}

\newpage

\end{document}